\documentclass[11pt,english]{article}
\usepackage[sc]{mathpazo}
\usepackage{geometry}
\usepackage{color}
\usepackage{array}
\usepackage{bbding}
\usepackage{multirow}
\usepackage[fleqn]{amsmath}
\usepackage{amssymb}
\usepackage{amsthm}
\usepackage{graphicx}
\usepackage{natbib}
\usepackage{lscape}
\usepackage[colorlinks,
linkcolor=blue, 
anchorcolor=blue, 
citecolor=blue,
]{hyperref}

\usepackage{comment}
\usepackage{bm}
\usepackage[title]{appendix}
\usepackage{longtable}
\usepackage{mathtools}
\usepackage{chngcntr}
\usepackage{authblk}
\usepackage{enumerate}
\makeatletter
\allowdisplaybreaks

\@ifundefined{showcaptionsetup}{}{%
 \PassOptionsToPackage{caption=false}{subfig}}
\usepackage{subfig}
\makeatother

\usepackage{babel}

\newtheorem{prop}{Proposition}

\newtheorem{defi}{Definition}

\DeclarePairedDelimiter{\abs}{\lvert}{\rvert}

\begin{document}\sloppy

\title{Rhythmic Control of Automated Traffic - Part I: Concept and Properties at Isolated Intersections}
\author{Xiangdong Chen\textsuperscript{a}\hspace{1em} Meng Li\textsuperscript{a}\hspace{1em} Xi Lin\textsuperscript{a}\hspace{1em} Yafeng Yin\textsuperscript{b,c}\hspace{1em} Fang He\textsuperscript{d}\footnote{Corresponding author. E-mail address: \textcolor{blue}{fanghe@tsinghua.edu.cn}.}}
\affil{\small\emph{\textsuperscript{a}Department of Civil Engineering, Tsinghua University, Beijing 100084, P.R. China}\normalsize}
\affil{\small\emph{\textsuperscript{b}Department of Civil and Environmental Engineering, University of Michigan, Ann Arbor, MI 48109, USA}\normalsize}
\affil{\small\emph{\textsuperscript{c}Department of Industrial and Operations Engineering, University of Michigan, Ann Arbor, MI 48109, USA}\normalsize}
\affil{\small\emph{\textsuperscript{d}Department of Industrial Engineering, Tsinghua University, Beijing 100084, P.R. China}\normalsize}

\date{\today}
\maketitle

\begin{abstract}
\noindent
Leveraging the accuracy and consistency of vehicle motion control enabled by the connected and automated vehicle technology, we propose the rhythmic control (RC) scheme that allows vehicles to pass through an intersection in a conflict-free manner with a preset rhythm. The rhythm enables vehicles to proceed at a constant speed without any stop. The RC is capable of breaking the limitation that right of way can only be allocated to non-conflicting movements at a time. It significantly improves the performance of intersection control for automated traffic. Moreover, the RC with a predetermined rhythm does not require intensive computational efforts to dynamically control vehicles, which may possibly lead to frequent accelerations or decelerations. Assuming stationary vehicle arrivals, we conduct theoretical investigation to show that RC can considerably increase intersection capacity and reduce vehicle delay. Finally, the performance of RC is tested in the simulations with both stationary and non-stationary vehicle arrivals at both symmetric and asymmetric intersections.  \par
\hfill\break%
\end{abstract}

\section{Introduction} \label{sec_intro}
Connected vehicle (CV) technologies are capable of building an interconnected network of moving vehicles and infrastructures, where vehicle-to-vehicle and vehicle-to-infrastructure communications can be realized in a collaborative and real-time manner. Fully automated vehicles (AV) are capable of gathering information, autonomously performing all driving functions, and monitoring roadway conditions for an entire trip (NHTSA, 2013). From a traffic operations perspective, the capabilities of AV technologies, integrated with CV systems, can further enable more responsive traffic controls, which imply a tremendous opportunity to more efficiently allocate right of way (ROW) at intersections. \par
In the literature, several studies have been conducted to improve intersection control under a fully connected and automated vehicle (CAV) environment, where an intersection controller can acquire the state information of each CAV including location, velocity and acceleration, and vehicles can access traffic information (e.g., signal timings or speed advisory) in real time. In a broad sense, related studies could be categorized into two groups, i.e., “signal-free” and “signalized” schemes. “Signal-free” schemes explicitly optimize the sequence of each CAV passing through an intersection (Levin and Rey, 2017; Lee and Park, 2012; Wu et al., 2012; Xu et al., 2018; Muller
et al., 2016), whereas “signal schemes” organize non-conflicting movements into groups/phases, then form platoons in each movement direction, and finally optimize the phase sequences (Li et al., 2014; Yu et al., 2018; Feng et al., 2018). To fully leverage the capabilities of CAV technologies, both schemes formulate the problem of passing sequence optimization as mathematical programs. However, exactly solving the optimization problems is undoubtedly burdensome, as the size of passing sequences grows exponentially with the number of incoming vehicles. For instance, as reported byLevin and Rey (2017),the proposed mixed integer linear program (MIP) that optimizes vehicle ordering at conflict points can only be solved in real time for up to 30 vehicles. Yu et al. (2018) proposed a MIP model to optimize vehicle trajectories and traffic signals, and it is showed that both constraint and binary variable exhibit quadratic growth with respect to the vehicle number in the worst cases. To satisfy a very time-constrained optimizations schedule required by a real-time control, some studies adopt heuristics such as ant colony system  (Wu et al., 2012), rolling horizon framework (Levin and Rey, 2017), and “first-come-first-served" (FCFS) strategy (Fajardo et al., 2011; Li et al., 2013). Nevertheless, it is always difficult to theoretically guarantee the solution quality of the above heuristics. As shown by Levin et al. (2016), compared to traffic signals, the FCFS policy leads to longer delays in some instances.Yu et al. (2019) recently conducted a theoretical queuing analysis of reservation-based policy according to a single conflict point; the results suggest that the FCFS is suboptimal in terms of throughput. Therefore, despite a growing interest in leveraging CAV technologies to improve intersection control, still lacking is a method that not only is scalable to real-world applications but also can considerably increase the intersection capacity and reduce vehicle delay.\par

To address the aforementioned research question, this paper suggests a new perspective in organizing traffic at intersections. Our concept is based on the fundamental observation at intersections that a complicated conflicting relationship is composed of a group of conflicting points among movements, and the crux for an intersection control is the temporal allocation of ROW among movements. The goal of an intersection control is to improve the intersection traffic efficiency, measured by throughput or vehicle delay, while avoiding collision at all conflicting points. Conceptually, the traffic throughput and delay of a movement is closely related to the percentage of time when it owns ROW. This can be intuitively explained through imagining that a link without any conflict points will always yield the maximum throughput and the minimum delay as the movement on it is always guaranteed $100\%$ ROW. At an intersection, we cannot provide $100\%$  ROW to a movement when it has conflicting flows. Nevertheless, if one can increase the number of movements that simultaneously own ROW at the intersection without any risk of collisions, then the intersection throughput would be increased. \par

The traditional way of intersection control is to organize non-conflicting movements into groups (phases) and sequentially assign ROW to one group at a time to avoid collisions. No matter how the allocation of green time among groups is optimized, the number of movements owning ROW simultaneously is limited by the number of non-conflicting movements in the group. Leveraging the accuracy and consistency of vehicle motion control enabled by the CAV technology, we propose a computationally cheap control that can break this limitation and allow more conflicting movements to simultaneously own ROW while ensuring safety. We call it as rhythmic control (RC) scheme. The RC scheme controls vehicles to pass through the intersection in a regularly recurring sequence, such that vehicles of any two intersecting lanes pass through the corresponding conflicting point in an alternating and conflict-free manner at a constant speed without any stop. We henceforth refer to this regularly recurring sequence as a rhythm. \par

Specifically, RC first untangles the intersection conflicts by designing a spatial layout of a few two-way conflicting points, and then enables CAVs to proceed within the intersection at a constant speed without any stop, by letting them follow a preset and coordinated rhythm. The rhythm assigns regularly recurring vehicle entry times for each lane such that the vehicles pass through each conflicting point in an alternating way. Such RC is capable of breaking the limitation that ROW can only be allocated to non-conflicting movements at a time, thus significantly improving the performance of intersection control. Moreover, the proposed RC scheme relies on a preset rhythm, rather than dynamically controlling the movements of vehicles/platoons to avoid conflicts in their space-time trajectories, as this latter method requires complicated computational efforts and possibly frequent accelerations or decelerations. Assuming stationary vehicle arrivals, we conduct theoretical investigation to show that RC can considerably increase the intersection capacity and reduce vehicle delay. For the former, we investigate an admissible demand set, defined as the set of demand rates that yield bounded queue lengths on all lanes.\par

This study mainly focuses on the traffic control at an isolated intersection. For dense urban areas, coordination between intersections is required to promote the mobility of the whole network. Thus, we further study the rhythmic control of CAVs on grid networks, which is arranged as a companion of this study (Part II); the detail can be found in Lin et al. (2020).

For the remaining parts of the paper, Section \ref{sec_motiv} provides a motivating example for the RC scheme. Section \ref{sec_RC} rigorously defines the RC scheme, and describes its collision-avoidance approach and implementation issues. Section \ref{sec_propoerties} derives the average vehicle delay and admissible demand set of RC. Section \ref{sec_numerical} conducts numerical experiments to compare RC with some typical existing control schemes in various settings. Finally, Section \ref{sec_conclusion} concludes the paper. \par

\section{Motivating example} \label{sec_motiv}

\begin{figure}[!ht]
	\centering
	\includegraphics[width=1\textwidth]{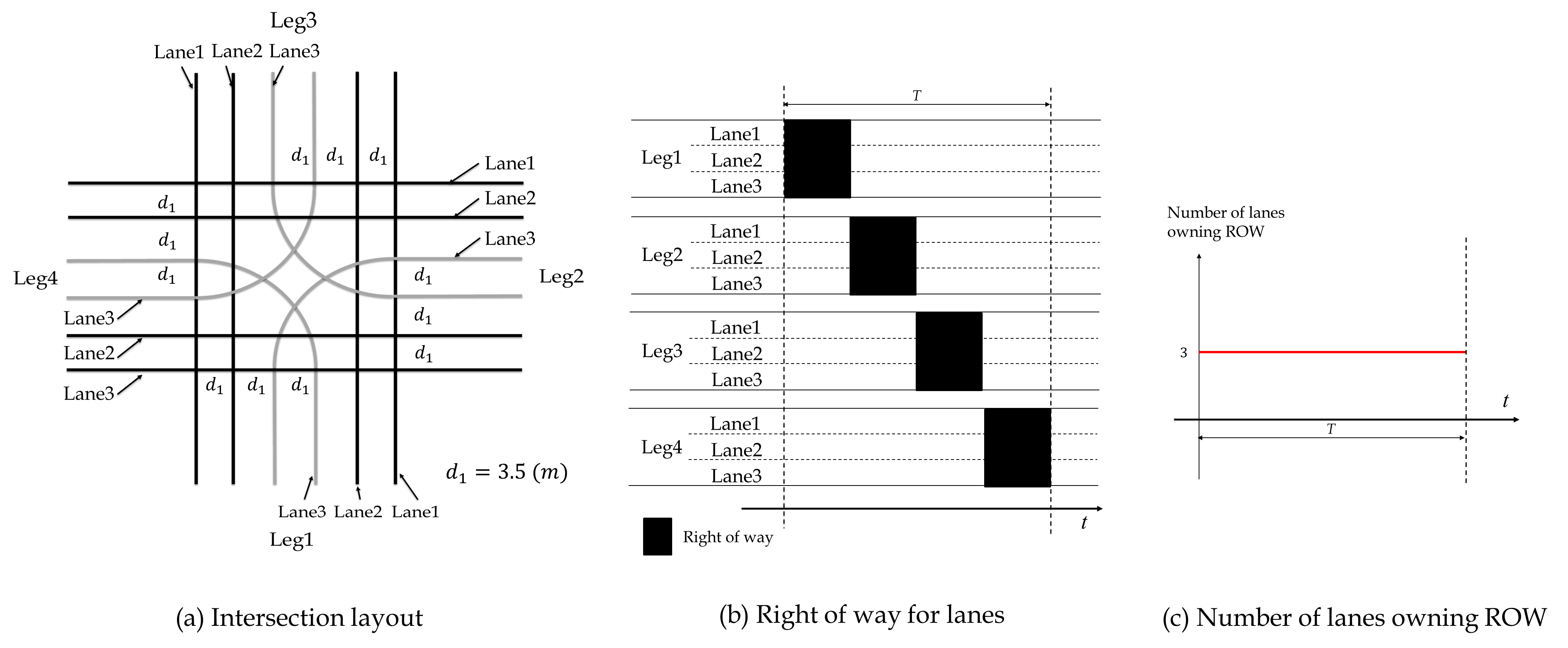}
	\caption[]{Right of way allocation under traditional TSC} 
	\label{fig_ROWTSC}
\end{figure}
\par

\begin{figure}[!ht]
	\centering
	\includegraphics[width=1\textwidth]{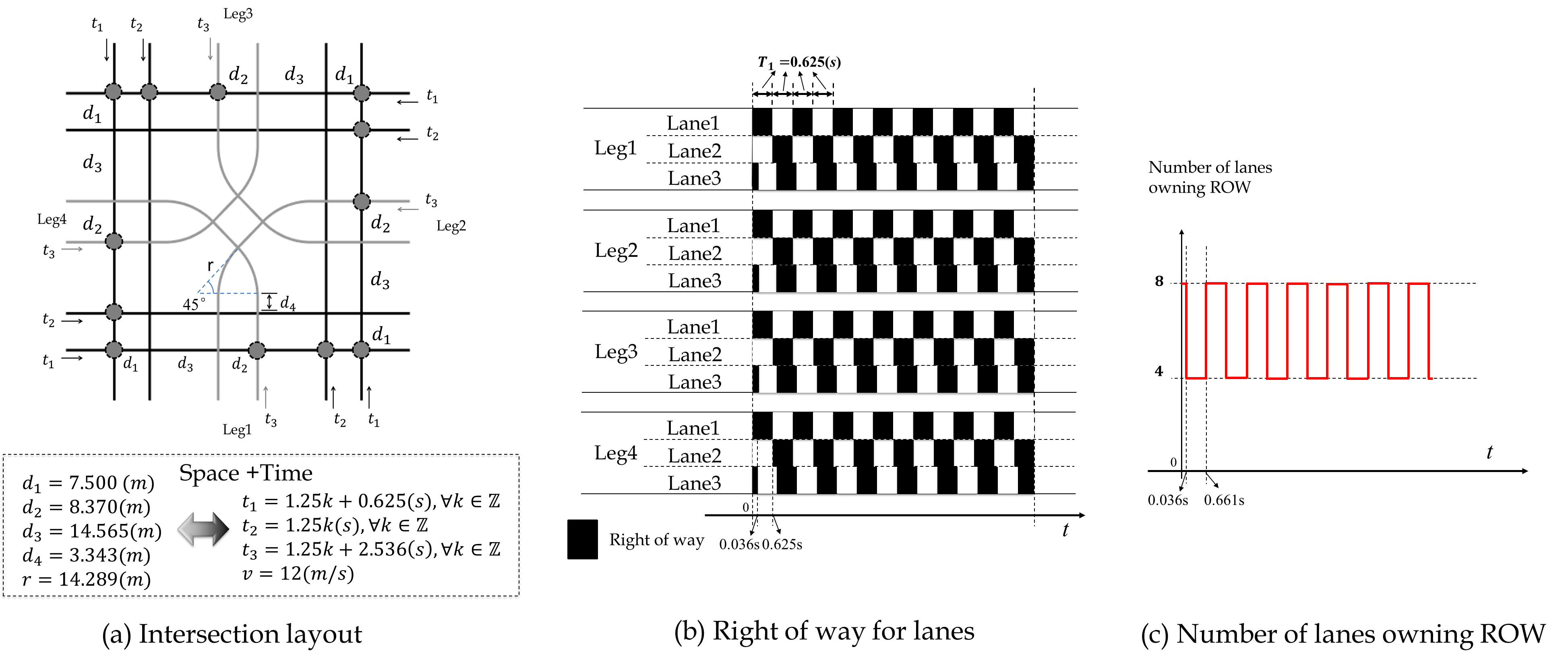}
	\caption[]{Right of way allocation under RC} 
	\label{fig_ROWRC}
\end{figure}
\par

\begin{figure}[!ht]
	\centering
	\subfloat[][]{\includegraphics[width=0.8\textwidth]{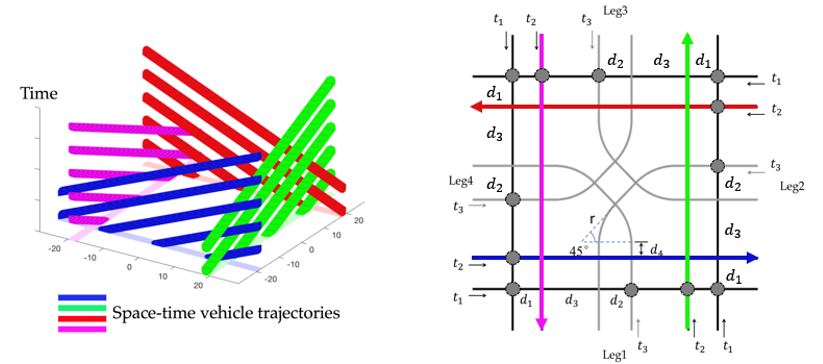}}\\
	\subfloat[][]{\includegraphics[width=0.8\textwidth]{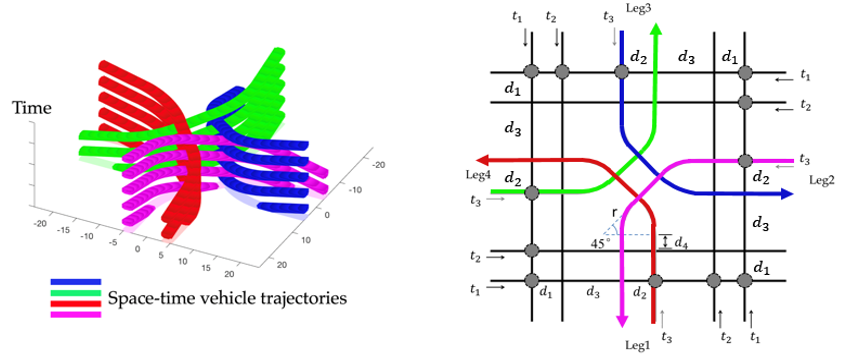}}\\
	\subfloat[][]{\includegraphics[width=0.8\textwidth]{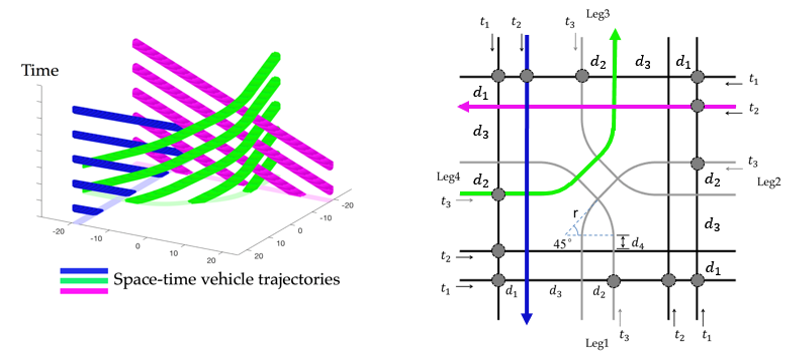}}
	\caption[]{Space-time trajectories at the isolated intersection in Figure 2(a).} 
	\label{spacetime_fig}
\end{figure}
\par

\noindent To better explain the concept of RC and demonstrate its advantage over traditional traffic signal control (TSC), consider a toy example of an isolated symmetric intersection with 3 lanes, where lanes 1, 2 and 3 of each leg are through and left-turn lanes, respectively, as shown in Fig.\ref{fig_ROWTSC}(a). Fig.\ref{fig_ROWTSC}(b) shows the temporal allocation of ROW among different lanes under a traditional TSC within a cycle, where the black bars represent the intervals of the lanes owning ROW. Fig.\ref{fig_ROWTSC}(c) depicts the number of lanes owning ROW over time. We can observe that no matter how signal timing is optimized, this number always equals the number of lanes in a signal phase, which cannot exceed the maximum number of lanes in non-conflicting movements. To break this limitation, we propose the rhythmic control (RC) scheme. Specifically, it first designs the spatial placement of conflicting points through re-designing the intersection layout, as shown in Fig.\ref{fig_ROWRC}(a). Then, it sets the rhythm, i.e., assigning the times of vehicles on each lane entering the first conflicting points, showed by the dotted circles in Fig.\ref{fig_ROWRC}(a). Under the logic of RC, vehicles on lanes 1, 2 and 3 of each leg periodically enter the corresponding first points at times $t_{1}=1.25k+0.625(s)$, $t_{2}=1.25k(s)$ and $t_{3}=1.25k+2.536(s), \forall k\in \mathbb{Z}$, respectively (the detailed computation approach of the parameters can be found in the definition of RC scheme of Section \ref{subsec_RCdescrib}). Then, they will cross the intersection at a constant speed $v=12 (m/s)$ without any stop.  It can be proved that this design totally resolves all the inter-vehicle conflicts. Fig.\ref{spacetime_fig} provides the designed vehicular space-time trajectories under the RC scheme within the intersection to demonstrate the conflict-free properties of RC. Fig.\ref{fig_ROWRC}(b) shows the temporal distribution of ROW on each lane under the RC scheme, whereas Fig.\ref{fig_ROWRC}(c) depicts the number of lanes owning ROW over time. It can be observed that the number of lanes owning ROW under the RC scheme is always higher than (sometimes even two times higher than) that under TSC, which intuitively confirms RC’s capability of utilizing the intersection more sufficiently. Furthermore, as shown in Fig.\ref{fig_ROWRC}(b), the length of each blank area between two consecutive black bars is only $0.625 s$, which contributes to reducing vehicle delays at the intersection. Finally, although Figs.\ref{fig_ROWRC}-\ref{spacetime_fig} demonstrate the performance improvement by an RC scheme for a particular intersection layout, it still remains as a major challenge to propose an RC scheme that could resolve all inter-vehicle conflicts at a generic intersection. \par

\section{Rhythmic control scheme} \label{sec_RC}

\subsection{Scheme Description} \label{subsec_RCdescrib}
\begin{figure}[!ht]
	\centering
	\includegraphics[width=0.7\textwidth]{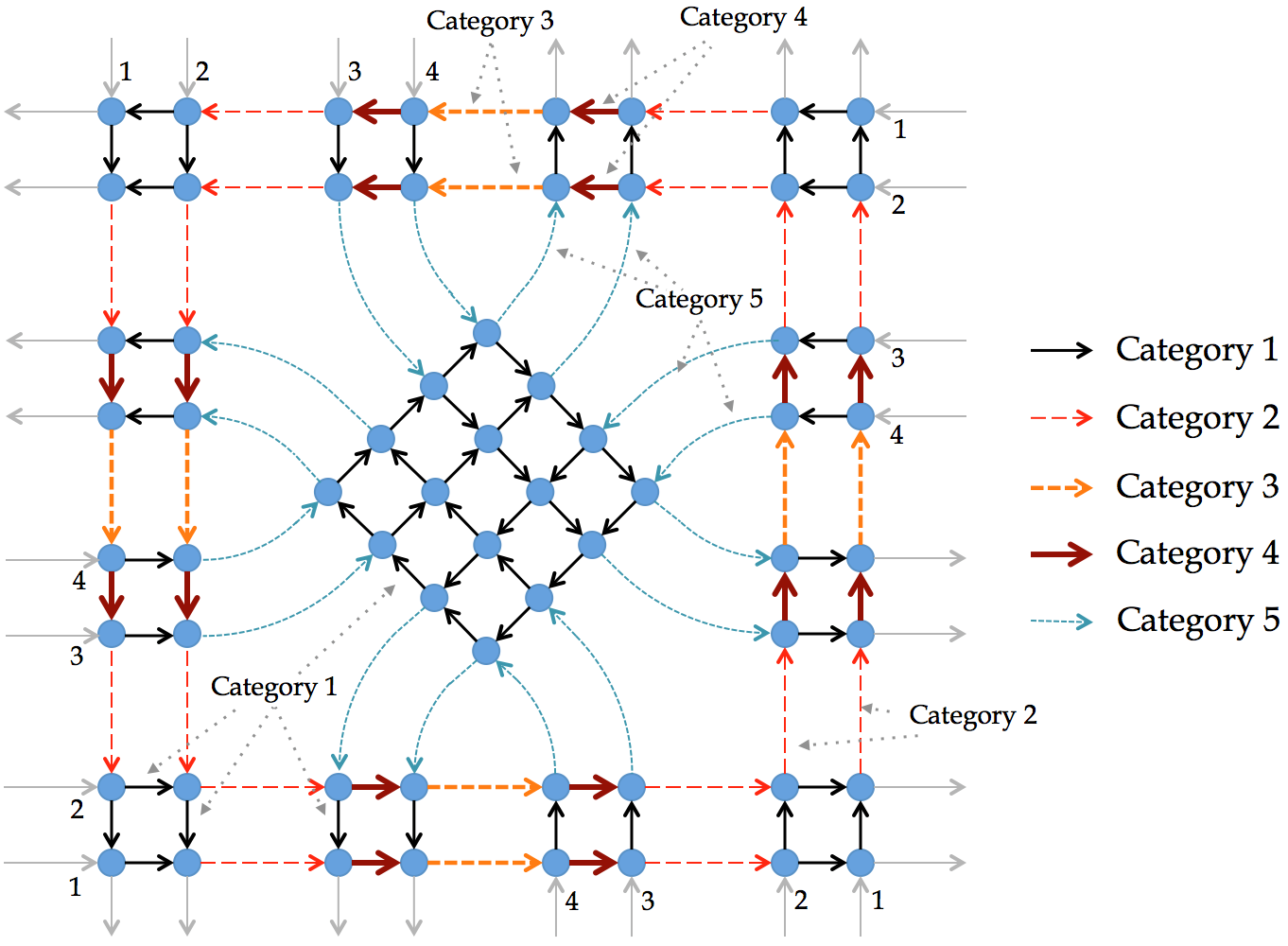}
	\caption[]{Categorization of lane segments at an intersection} 
	\label{lanesegment_fig}
\end{figure}
\par
\noindent We utilize a symmetric four-leg intersection, shown in Fig.\ref{lanesegment_fig}, to demonstrate the execution process of the RC. In Fig. \ref{lanesegment_fig}, an arrow represents a lane segment (defined as the section of a lane between two consecutive conflict points), and a circle represents a conflicting point. The numbering system of the four legs is in a counter-clockwise manner, and the lanes on each leg are numbered from the border line to the center line as $1, 2, …, n_{s}, n_{s+1}, …, n_{s}+n_{l}$, where $n_{s}$ and $n_{l}$ are the number of through and left-turn lanes ($n_{s}=2$ and $n_{l}=2$ for the intersection illustrated in Fig.\ref{lanesegment_fig}), respectively. Right-turn lanes are not considered because they do not directly conflict with other lanes. Moreover, to untangle the intersection conflicts, we do not allow the mixed-use situations of left-turn/through lanes. All lane segments are categorized into five types according to the geometric features and corresponding preset travel time. Specifically, Category 1 segments are those that connect two through lanes or those that lie in the central zone with conflicting left-turn lanes; Category 2 segments are those located on through lanes and connect one through lane and one left-turn lane; Category 3 segments are those on through lanes and connect two left-turn lanes from different legs; Category 4 segments are those on through lanes and connecting two left-turn lanes from the same legs; Category 5 segments are those on left-turn lanes that connect one through lane and one left-turn lane. Readers can refer to Fig.\ref{lanesegment_fig} for the categorization. The sets of segments of Categories 1–4 are denoted by $\mathcal{S}_{i}, i=1,2,3,4$, respectively. The  set of segments of Category 5 on lane number $l\in\{n_{s}+1……,n_{s}+n_{l}\}$ is denoted by $\mathcal{S}_5^l$ as the lane number $l$ is related to the segment length. The vehicle travel time on any segment in $\mathcal{S}_i, i=1,2,3,4$ is denoted $T_i$, and the travel time on any segment in $\mathcal{S}_5^l,l\in\{n_{s}+1……,n_{s}+n_{l}\}$ is $T_5^l$. Note that in the RC scheme, both $T_i$ and $T_5^l$ are preset and fixed travel times, which are determined by the preset vehicle traveling speed in the conflict zone and intersection geometric layout; they do not change in real time. Notations used in this paper are listed in Appendix A.\par

With the foregoing, the RC scheme is shown as follows. Here, it is defined that a vehicle enters an intersection if it touches the first conflict point on its lane, as shown in Fig.\ref{lanesegment_fig}. \\ \par

\noindent\underline{\textbf{{Rhythmic Control}}}

\noindent $\bullet$ A benchmark time (e.g., 0:00:00 on a day) is set as $t=0$; \par
\noindent $\bullet$ Vehicles on through lane $l\in\{{1,2,…,n_s}\}$ enter the intersection at time:
\begin{align*}
   t=\begin{cases}(2k+1)T_1, \text{ if $l$ is odd}\\(2k)T_1,\text{ if $l$ is even} \end{cases},k\in \mathbb{Z};
\end{align*}
\noindent $\bullet$ Vehicles on left-turn lane $\hat{l}\in\{{n_s + 1,...,n_s + n_l}\}$ enter the intersection at time:
\begin{align*}
    t = \begin{cases}(2k + n_s - 1)T_1+(2n_l)T_4+T_2+T_3,\text{ if ($\hat{l}-n_s$) is odd}\\ (2k + n_s - 1)T_1+(2n_l-1)T_4+T_2+T_3,\text{ if ($\hat{l}-n_s$) is even}\end{cases},k\in\mathbb{Z},
\end{align*}
where $\mathbb{Z}$ represents the set of all integers.\\ 

As shown above, for any approaching lane, RC only sets the associated entry time points and requires no real-time trajectory control in the conflict zone. It is only necessary for vehicles to adjust their trajectories to satisfy the corresponding entry time points before reaching the intersection; consequently, the computational load for the control center is substantially relieved. Although the control logic of RC is considerably simple, it exhibits some promising properties, which are gradually introduced in the following sections. 

\subsection{Collision avoidance} \label{subsec_CollisionAvoidance}
\noindent
In this section, it is shown that RC can guarantee collision avoidance with appropriate design, as verified by Proposition \ref{Prop_collisionfree}. In the following discussions, parameter $T$ is defined as the minimum time gap of two vehicles from conflicting lanes consecutively passing through the conflicting point. If the time gap is no less than $T$, then collision can be avoided at the conflicting point.\par

\begin{prop}\label{Prop_collisionfree}
RC is collision-free if the following conditions hold:
\begin{enumerate}
    \item [(1)]$T_1=T$;
    \item [(2)]$T_4=(2k_0+1) T_1,k_0\in \mathbb{N}$;
    \item [(3)]$2T_2+T_3=(2k_0^{'}+1) T_1,k_0^{'}\in \mathbb{N}$;
    \item [(4)]$2T_5^l+T_3=(2k_0^{''}+1) T_1,k_0^{''}\in \mathbb{N},l\in \{n_s+1,…,n_s+n_l\}$;
    \item [(5)]$T_5^i-T_5^j=2k_0^{'''}T_1,k_0^{'''}\in \mathbb{N},i,j\in \{n_s+1,…,n_s+n_l\},i<j$;
\end{enumerate}
Here, $\mathbb{N}$ represents the set of all non-negative integers.
\end{prop} 

\begin{proof}
See Appendix B.
\end{proof}

For an intersection, it is straightforward to verify that there exist infinite combinations of $T_i$ and $T_5^l$ that satisfy Conditions (2)–(5). Once the basic time interval $T_1$ is given, for any length of the segments, based on Conditions (2)–(5), we can always find appropriate $k_0$, $k_0^{'}$, $k_0^{''}$, $k_0^{'''}$ to determine the values of $T_2$, $T_3$, $T_4$ and $T_5^l$ such that the vehicle speeds are within a reasonable range (in the view of vehicle dynamics). Therefore, the segment travel time is not quite sensitive to the intersection configurations. An example of an isolated intersection that satisfies these conditions is given in Fig.\ref{fig_ROWRC}(a).

To this end, we have described the scheme of RC and shown the collision avoidance only at symmetric intersections. For asymmetric intersections with different numbers of lanes on two intersecting roads, the proposed scheme can be readily generalized, as an asymmetric intersection can be treated as a symmetric intersection with some virtual lanes (see Fig.\ref{asymmetric_fig}). Specifically, the asymmetric intersection with these actual lanes in Fig.\ref{lanesegment_fig} is guaranteed to be collision-free by designing the entry time spots on both actual and virtual lanes as per the simple procedure in Section \ref{subsec_RCdescrib}. To highlight the effectiveness of RC, the test on its performance at asymmetric intersections is presented in Section \ref{sec_numerical}.

\begin{figure}[!ht]
	\centering
	\includegraphics[width=0.7\textwidth]{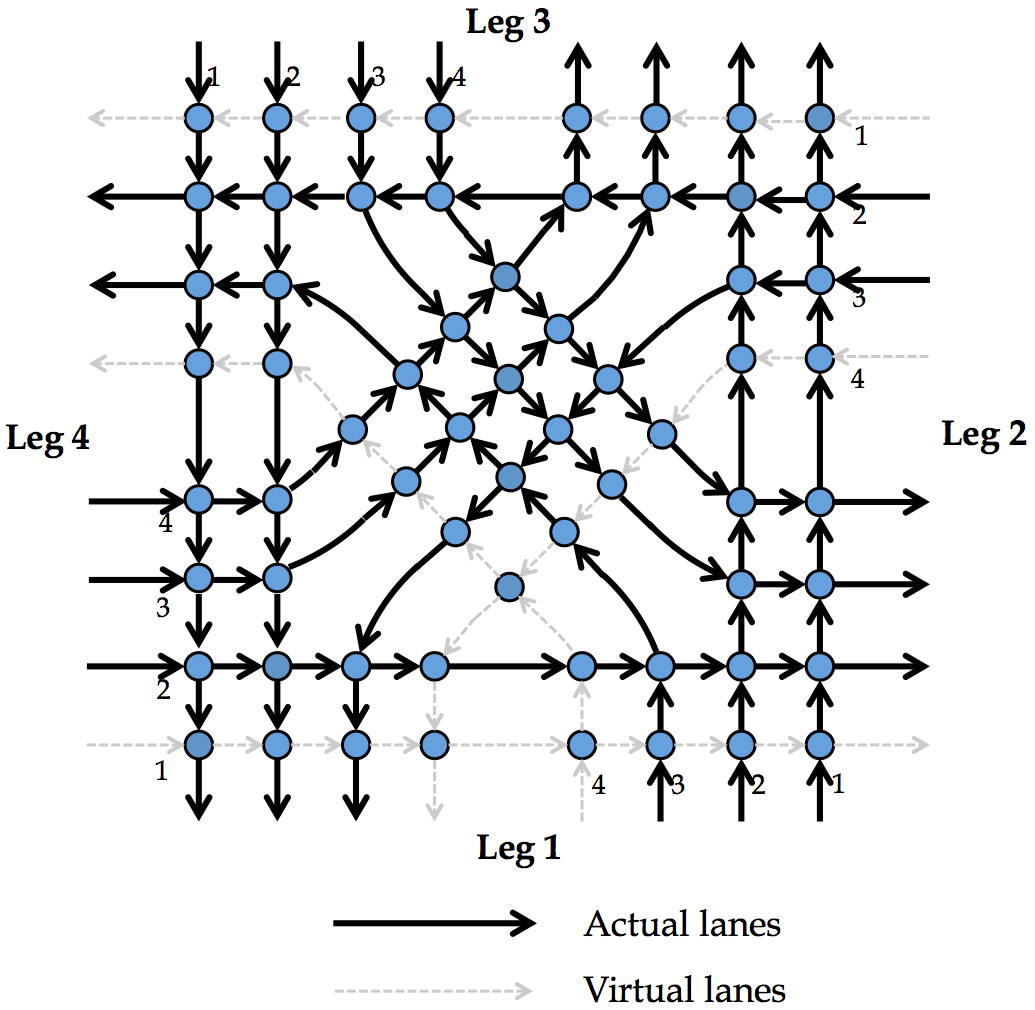}
	\caption[]{Illustration of an asymmetric intersection} 
	\label{asymmetric_fig}
\end{figure}
\par

\subsection{Implementing RC in fully connected automated environment} \label{subsec_Implementing}
\noindent
As described, RC requires an ordered and highly precise vehicle movement; this suggests that a fully CAV environment is needed. In this section, some related issues on the implementation of RC are discussed. To regulate the trajectories of CAVs as per the RC policy, a road segment immediately before an intersection is divided into two parts, as shown in Fig.\ref{zonedivision_fig}. The vehicles travel from the adjustment zone to the conflicting zone; the length of adjustment zone can be set reasonably large to allow sufficient room for trajectory adjustment (e.g., 100 m as that in Yu et al. (2019)). The typical setting of RC in a fully CAV environment is described as follows. 

\begin{figure}[!ht]
	\centering
	\includegraphics[width=0.8\textwidth]{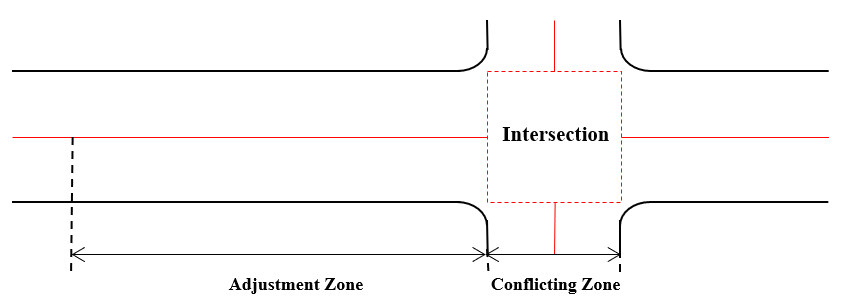}
	\caption[]{Zone division} 
	\label{zonedivision_fig}
\end{figure}
\par

\begin{enumerate}
\item [i]Once a CAV enters the adjustment zone, its real-time speed and acceleration are uploaded to a centralized or roadside controller, and the controller will take over (or guide) the vehicle movement in the zone. This controller is responsible for adjusting the trajectory of the vehicle by a set of simple and computationally inexpensive rules such that the vehicle can reach the conflict zone with a preset speed at a specified timing.  
\item [ii]Within the conflicting zone, the vehicle is fully controlled by the centralized controller and follows a pre-given and fixed speed until it leaves the intersection. 
\end{enumerate}

During the entire process, it is only necessary for the controller to adjust vehicle speeds by a set of simple rules in the adjustment zone and to control the vehicles to follow the pre-given and fixed speed in the conflicting zone; these yield a considerably limited computational load for the controller. \par
The minimum allowable distance between any two vehicles for ensuring safety is set as $\delta$. It is assumed that all vehicles have identical length ($L$) and width ($w$). Note that the identical length assumption is not necessarily restrictive. In fact, we can consider $L$ and $w$ to be the maximum length and width of all vehicles, respectively, without affecting the main results of this study. Furthermore, if a vehicle with a length larger than $L$ enters an intersection, then the centralized controller can close some entry time spots with potential collisions for this vehicle to ensure safety. There is a maximum speed ($v_m$) and a maximum absolute value of acceleration and deceleration rates ($a_m$). In the conflicting zone, vehicles travel at a speed $v_m$. All vehicles are located on their target lanes before entering the adjustment zone; thus, there is no lane-changing behavior in the adjustment and conflicting zones. All lanes intersect in a perpendicular way. \par

In this setting, to physically avoid collisions, parameter $T$ needs to satisfy $T\geq\frac{L+w+\sqrt{2} \delta}{v_m}$. The details of deriving this inequality are shown in Appendix C. Recall that $T$ represents the minimum time gap of two vehicles from conflicting lanes consecutively passing through the conflicting point. Combining $T\geq\frac{(L+w+\sqrt{2} \delta)}{v_m}$ and $T_1=T$ (Condition (1) in Proposition \ref{Prop_collisionfree}) yields $v_mT_1\geq{L+w+\sqrt{2} \delta}$, suggesting that the interval between two parallel through lanes or left-turn lanes must exceed a minimum distance, and the minimum distance could be generally larger than the ordinary lane interval (e.g., $3.5 m$). This observation implies the possible necessity to redesign the geometry of an intersection for implementing RC. Fig.\ref{intersectionlayout_fig} provides an example. Note that even though the lane interval is enlarged, the lane width can be kept unchanged. Accordingly, it is possible to utilize the space between adjacent lanes for other functionalities.Note that the proposed RC is more suitable to be implemented in places where major roads intersect. Such major intersections are often the bottlenecks of local road networks, and thus improving their associated capacities could relieve urban congestion to a great extent. For dense urban areas, we propose a grid-network rhythmic control framework, where the coordination between intersections is required to promote the mobility of the whole network (Lin et al., 2020).\par

\begin{figure}[!ht]
	\centering
	\includegraphics[width=0.7\textwidth]{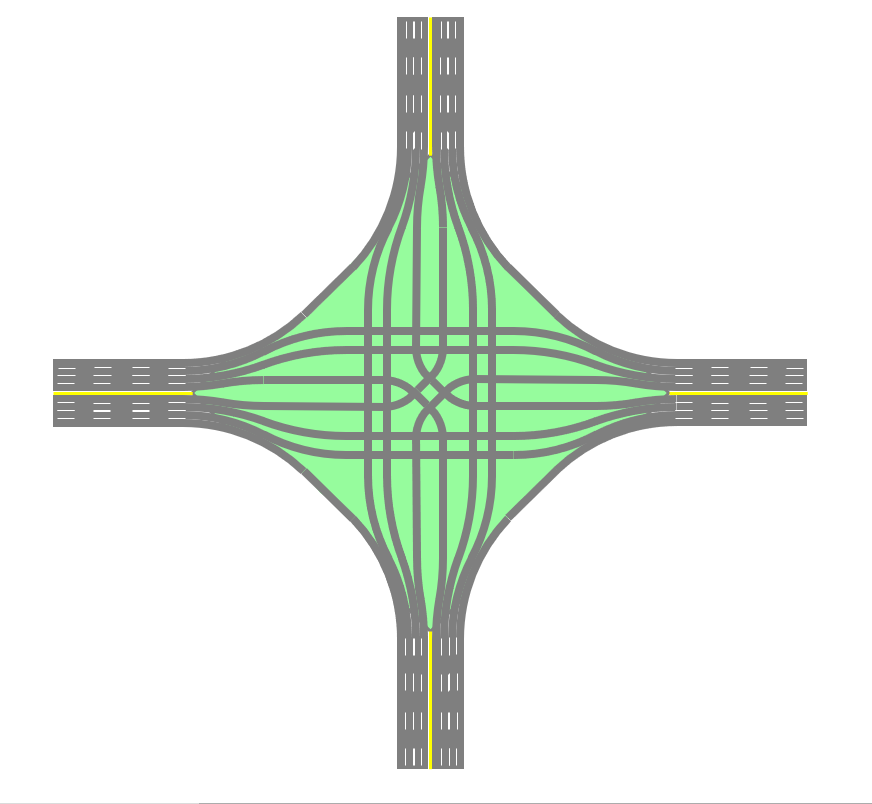}
	\caption[]{Enlarged lane intervals with redesigned intersection layout} 
	\label{intersectionlayout_fig}
\end{figure}
\par
Finally, the implementation of RC requires that vehicles enter an intersection at designated time spots and speeds. To achieve this goal, the trajectories of vehicles in the adjustment zones must be adjusted before entering the intersection. Appendix D introduces a set of parsimonious rules for trajectory adjustment that avoids complicated online computation or optimization. Note that the trajectory adjustment in the adjustment zone is essentially an optimal control problem (Feng et al., 2018), which is beyond the scope of this study. Considering that the proposed RC is devoted to reducing the computational efforts of a controller to the extent possible, only a set of computationally inexpensive rules are provided to identify reasonable vehicle trajectories.

\section{Properties of RC under stochastic vehicle arrivals} \label{sec_propoerties}

\noindent
This section presents the theoretical performances of RC under more realistic settings, i.e., with stochastic vehicle arrivals. The analyses are focused on average vehicle delay (Section \ref{subsec_delay}) and admissible demand set (Section \ref{subsec_AdmissibleDemand}).  Throughout this section, we set the vehicle arrival to be stationary with a rate of $\theta > 0$, but we do not impose any specific stochastic process (e.g., Poisson's process) that the arrivals must follow; the only assumption made is that the vehicle arrivals of any two entrance intervals are independent of each other. With the independence setting, we can adopt Markov chain approach to derive the target results.

\subsection{Average vehicle delay} \label{subsec_delay}

\noindent Based on the RC scheme, the benchmark time for each lane can be reset without loss of generality such that the intersection-entry time points of vehicles are $0,2T_1,4T_1,…,2kT_1,…,$ where $T_{1}={\frac{(L+w+\sqrt{2} \delta)}{v_m}}$; the cruising time in the adjustment zone is ignored. To model the stochastic vehicle queuing length in the adjustment zone of an intersection under the RC scheme in a parsimonious manner, a discrete Markov chain model is employed to specify the number of queuing vehicles immediately before each intersection-entry time point. The number of queuing vehicles immediately before the $k^{th}$ entry time point is denoted as $w_k$. In this Markov chain, $w_k$ can be interpreted as the state variable at step $k$; the state space is $\mathbb{N}$. Then, the initial state and state transition probability can be defined as follows:
\begin{align*}
& \mathbb{P}(w_0 = 0) = 1 \\
& \mathbb{P}(w_{k+1} = b | w_k = 0) = P_{b}^{\theta} & k \ge 0, b \ge 0 \\
& \mathbb{P}(w_{k+1} = a + b | w_k = a) = P_{b+1}^{\theta} & k \ge 0, a \ge 1, b \ge -1
\end{align*}

\noindent where $P_{i}^{\theta}, i \in \mathbb{N}$ is the probability that $i$ vehicles arrive between two consecutive entry time points with the intensity, $\theta$; it must satisfy $\sum_{i=0}^{+\infty} P_{i}^{\theta} = 1$ and $\sum_{i=0}^{+\infty} i P_{i}^{\theta} = 2 \theta T_1$, where the latter indicates that the expected number of vehicles arriving during one interval is $2 \theta T_1$. Then, the state transition matrix can be defined as $\boldsymbol{\pi}^{\theta} = [\pi_{ij}^{\theta}], i \ge 0, j \ge 0$, where

\begin{align*}
& \pi_{ij}^{\theta} = \left\{ 
\begin{matrix}
P_i^{\theta} & j = 0 \\
P_{i-j+1}^{\theta} & j \ge 1, i \ge j - 1 \\
0 & \mathrm{otherwise}
\end{matrix}
\right.
\end{align*}
Accordingly, the steady-state equations of this Markov chain can be written as follows:

\begin{align}
& p_i^{\theta} = \sum_{j=0}^{+\infty} p_j^{\theta} \pi_{ij}^{\theta} & \forall i \in \mathbb{N} \label{stationary_eq_1}\\
& \sum_{j=0}^{+\infty} p_j^{\theta} = 1 \label{stationary_eq_2}
\end{align}
where $p_i^{\theta}$ is the steady-state probability that $i$ vehicles are in the queue immediately before an entry time point under arrival rate $\theta$. Eq.(\ref{stationary_eq_1}) can be further simplified to take the form of the following recurrent equations:

\begin{align}
& p_1^{\theta} = \frac{1 - P_0^{\theta}}{P_0^{\theta}}p_0^{\theta} \label{recurrent_eq_1}\\
& p_2^{\theta} = \frac{1 - P_0^{\theta} - P_1^{\theta}}{(P_0^{\theta})^2} p_0^{\theta} \label{recurrent_eq_2} \\
& p_{i+1}^{\theta} = \frac{1}{P_0^{\theta}} \left[ (1 - P_1^{\theta}) p_i^{\theta} - \sum_{j=1}^{i-1} P_{i+1-j}^{\theta} p_j^{\theta} - P_i^{\theta}p_0^{\theta} \right] \label{recurrent_eq_3}
\end{align}
By solving the above linear system, we can then obtain the steady-state probability distribution $p_j^{\theta}, j \in \mathbb{N}$. With the state probability, we then derive the formulation for the expected queuing delay. By Little's law (Little, 1961), in a queuing system, the average delay can be computed as the average queue length divided by arrival rate, therefore we have:
\
\begin{align}
\overline{D}(\theta) &= \frac{1}{\theta} \sum_{n=1}^{+\infty} \frac{1}{2}(n + n - 1) p_n^{\theta} = \frac{1}{\theta} \left( \frac{1}{2} \sum_{n=1}^{+\infty} p_n^{\theta} +  \sum_{n=1}^{+\infty} (n-1) p_n^{\theta} \right)
\end{align}
Here, the average queue length is computed by $\sum_{n=1}^{+\infty} \frac{1}{2}(n + n - 1) p_n^{\theta}$; this is because $p_n^{\theta}$ is the probability that $n$ vehicle is waiting \textit{right before} an entry time point, and it will soon become $\max(n-1,0)$ after the entry time; therefore we should average them. From Eqs.(\ref{recurrent_eq_1})-(\ref{recurrent_eq_3}), we can acquire the following result: $\sum_{k=2}^{+\infty} (k-1) p_k^{\theta} = -\frac{\sum_{l=1}^{+\infty} l P_l^{\theta}}{P_0^{\theta}} p_0^{\theta} + \sum_{j=1}^{+\infty} \frac{j - \sum_{l=1}^{+\infty} (l+j-1) P_l^{\theta}}{P_0^{\theta}} p_j^{\theta} $. Note that because $\sum_{i=0}^{+\infty} P_{i}^{\theta} = 1$ and $\sum_{i=0}^{+\infty} i P_{i}^{\theta} = 2 \theta T_1$, we obtain: $ \sum_{k=2}^{+\infty} (k-1) p_k^{\theta} = - \frac{2 \theta T_1}{P_0^{\theta}} p_0^{\theta} + \sum_{j=1}^{+\infty} \frac{1 - 2 \theta T_1}{P_0^{\theta}} p_j^{\theta} + \sum_{k=2}^{+\infty} (k-1) p_k^{\theta}$, which leads to $p_0^{\theta} = 1 - 2 \theta T_1$. Therefore, we acquire: 

\begin{align} \label{avg_delay_eq}
\overline{D}(\theta) &= \frac{1}{\theta} \left( \frac{1}{2} \sum_{n=1}^{+\infty} p_n^{\theta} +  \sum_{n=1}^{+\infty} (n-1) p_n^{\theta} \right) = T_1 + \frac{1}{\theta} \sum_{n=1}^{+\infty} (n-1) p_n^{\theta}
\end{align}
The following proposition provides an upper bound for the average vehicle delay $\overline{D}(\theta)$ under a rather mild assumption; this assumption requires that at most two vehicles arrive in an interval between two consecutive entry time points. This is a reasonable setting because the time interval $2 T_1$ is only a bit larger than the minimum time required for two vehicles to pass. Details are stated below.

\begin{prop} \label{delay_bound_prop}
When $2 \theta T_1 < 1$, if $P_{i}^{\theta} = 0$ for all $i > 2$, then $\overline{D}(\theta) \le T_1 + \frac{T_1}{1 - 2\theta T_1}$.
\end{prop}

The proof is presented in Appendix B. Note that the conclusion in Proposition \ref{delay_bound_prop} applies to any vehicle arrival process as long as $P_{i}^{\theta} = 0$ for all $i > 2$, including some scenarios where vehicles arrive in bunches. It suggests that RC is robust to a certain level of arrival fluctuations. However, the result does not apply to the scenarios where there is periodical fluctuation with a relatively long cycle (e.g., the downstream flow of a traffic light); this is because the independence assumption is violated there. We will demonstrate the average vehicle delays under those scenarios with numerical experiments.\par

Next, we identify a special case that leads to an analytical solution of the average vehicle delay, i.e., Eq.(\ref{avg_delay_eq}); the case requires the vehicle arrivals follow Poisson's process. In this context, the vehicle queuing process can be regarded as a M/D/1 system with vacation time, and the vacation time is $2T_1$ as in the period between two entry time points, no vehicle is allowed to pass through in one lane. From Lee (1989), the average waiting time of the customers can be calculated by:
\begin{align*}
    & W=\frac{\lambda/\mu^2}{2(1-\lambda/\mu)}+\frac{v}{2} & \frac{\lambda}{\mu}<1 \text{ and } \lambda\times\mu<1
\end{align*}

where $\lambda$, $\mu$ and $v$ denote arrival rate, departure rate and vacation time, respectively. Therefore, the average delay for vehicles under the RC scheme can be derived accordingly as:
\begin{align} \label{delay_poisson_eq}
    & \overline{D}(\theta)=\frac{\theta(2T_1)^2}{2(1-2\theta T_1)}+\frac{2T_1}{2}=\frac{T_1}{1-2\theta T_1} & 2\theta T_1<1
\end{align}

The above results provide abundant insights on the average vehicle delay under the RC scheme. Recall that $2\theta T_1$ is the average number of arrival vehicles during one time window. We can conclude from Proposition \ref{delay_bound_prop} that the queue is always bounded when the traffic volume is no more than one vehicle per $2T_1$ unit of time. Furthermore, given a fixed $T_1$, the average vehicle delay  $\overline{D}(\theta)$ exhibits a shape of “reversed inverse proportion function” as shown in Fig.\ref{vehicledelay_fig} under Poisson's vehicle arrival setting. This type of function starts with $T_1$ at $\theta=0$, and then increases slowly with $\theta$ within a large range; its slope becomes steep only when $2\theta T_1$ is approaching one. As observed in Fig.\ref{vehicledelay_fig}, the average vehicle delay is almost negligible (i.e., less than five seconds) when $2\theta T_1$ is not so close to one. The average delay can be further reduced by decreasing $T_1$.

\begin{figure}[!ht]
	\centering
	\includegraphics[width=0.65\textwidth]{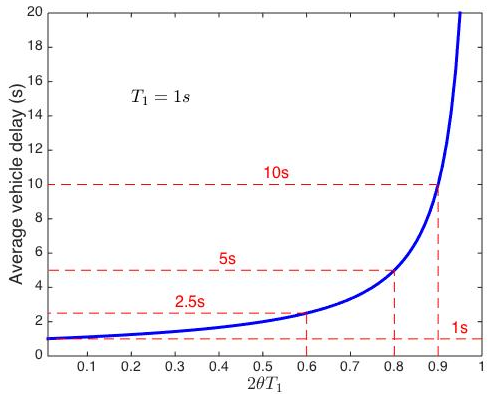}
	\caption[]{Average vehicle delay under Poisson's arrival when $T_1=1s$}
	\label{vehicledelay_fig}
\end{figure}
\par

\subsection{Admissible demand set} \label{subsec_AdmissibleDemand}
\noindent
The definition of the term “admissible demand set” is presented below.

\begin{defi}
The admissible demand set for a control policy is the set of demand rates of all approaches such that the queue lengths are bounded on all lanes by the control policy.
\end{defi}

Proposition \ref{delay_bound_prop} proves the boundedness of vehicle queue on a lane under the mild assumption that $P_i^{\theta} = 0$ for all $i > 2$. In this section, we will provide a more general result on the admissible demand set of RC that holds for all possible vehicle arrivals. The details are presented below.

\begin{prop} \label{admissible_prop}
Under any vehicle arrivals satisfying $\sum_{l=1}^{+\infty} l^2 P_{l}^{\theta} < +\infty$, the admissible demand rate of RC is given by
\begin{align}\label{admissible demand_eq}
& \theta_i < \frac{1}{2T_1} =  \frac{v_m}{2(L+w+\sqrt{2}\delta)} & \forall i \in \{1,2,...,8\}
\end{align}
where $\theta_i$ denotes the demand rate per lane for approach $i$.
\end{prop}

The proof is provided in Appendix B. Note that $\sum_{l=1}^{+\infty} l^2 P_{l}^{\theta} < +\infty$ always holds in reality since the number of vehicles arriving within an interval is always finite. From Proposition \ref{admissible_prop}, we see that the admissible demand set is a hypercube (excluding the surfaces).

In the following, we adopt some realistic settings and provide a quantitative illustration of the admissible demand set of RC. Let the maximum speed at an intersection be $v_m=10m/s$ and vehicle size be $L=4.5m$ and $w=2m$. Furthermore, the minimum allowable distance between two vehicles is $\delta=1m$,  based on the results of Rajamani and Shladover (2001). Note that the minimum allowable distance is always set to be smaller than that of human-driven vehicles, considering that CAV technologies could reduce the reaction time of vehicles and V2V technologies could assist vehicles to communicate with each other and achieve cooperative driving. Then, the admissible demand rate per lane is $\frac{v_m}{2(L+w+\sqrt{2}\delta)}\approx0.63 \ veh/s\approx2,274 \ veh/h$, i.e., a lane at the intersection under RC can accommodate approximately $2,270 \ veh/h$, which is even larger than the maximum lane capacity of a motorway in the current manually driven vehicle environment, i.e., approximately $1,800–2,400 \ veh/h$ per lane (Immers and Logghe, 2002). Again, note that many lanes at the intersection can simultaneously active the RC; this implies that RC can considerably improve the intersection throughput compared to the TSC.

\section{Numerical experiments} \label{sec_numerical}

\subsection{Basic settings} \label{subsec_Basicsetting}

\noindent
This section presents simulation results to demonstrate the performances of RC in various scenarios. The benchmarks include the TSC and reservation-based scheme in the first-come-first-serve protocol (FCFS; Dresner and Stone, 2004). Note that the TSC is also implemented in a fully CAV environment; thus, vehicles can pass through an intersection within a green phase at a considerably higher rate than human-driven vehicles.

In all the tests below, we set $\delta=1 m$, vehicle length, $L=4.5 m$, and width, $w=2 m$; the speed at the intersection is $v_m=10 m/s$. It is assumed that RC exhibits a systematic delay of one second in all vehicles because of the redesign of the intersection layout (Section \ref{subsec_Implementing}). The phase transition time loss is set as two seconds for the TSC; with this loss, Webster’s method (Webster, 1958) can be adopted to calculate the timing allocation and cycle length of the TSC. The minimum allowable phase timing is $g_{min}=4 s$, and the maximum allowable cycle length for the TSC is $180 s$. The time interval for two consecutive reservations in the FCFS is set to $0.1 s$. All simulation programs are coded using MATLAB 2018a on a laptop with 2.4-GHz Intel Core and 4-GB RAM.

\subsection{Tests at a symmetric intersection} \label{subsec_testonsym}

\noindent
First, simulation tests at a symmetric intersection with $n_s=3$ and $n_l=2$ are performed. Three different scenarios, i.e., balanced demand scenario ($\mathbf{d}_b$), imbalanced scenario ($\mathbf{d}_i$), and highly imbalanced scenario ($\mathbf{d}_h$) are tested; their values are given by
\begin{align*}
    & \mathbf{d}_b = \alpha\times [1300,1300,1300,1300,1100,1100,1100,1100] \\
    & \mathbf{d}_i = \alpha\times[1600,1600,1600,1600,800,800,800,800] \\
    & \mathbf{d}_h = \alpha\times[2600,1400,1400,1400,400,400,400,400]
\end{align*}
where $\alpha$ is a scaling factor that represents the traffic demand level. In each demand vector, the first four elements represent four through-approach demands (per lane), and the last four elements are the four left-turn-approach demands (per lane); in $\mathbf{d}_i$, the demands per lane in the four through approaches are larger than those in the four left-turn approaches; and in $\mathbf{d}_h$, the demands per lane in one approach are distinctively larger than the others. Note that the sum of maximum demands in all four phases is set as a constant, i.e., $4,800$, for all three scenarios. To achieve more comprehensive comparison results, the performances are tested in both stationary and non-stationary vehicle arrivals.

\subsubsection{Stationary vehicle arrivals} \label{subsubsec_stationary}

\noindent
The stationary vehicle arrival scenario implies that the vehicle arrivals on each lane follow a time-invariant process with a headway that obeys a shifted exponential distribution, i.e., the intensity of vehicle arrivals in all approaches does not vary with time,  and they are evenly divided onto all lanes compatible with the associated movements. The simulation results in all scenarios are given in Figs.\ref{stat_1_fig}(a)-(c), where the left column denotes average delay and the right column denotes vehicle throughput (defined as the total number of vehicles passed through the intersection within the simulation time).
\begin{figure}[!ht]
	\centering
	\subfloat[][Balanced demand scenario]{\includegraphics[width=0.8\textwidth]{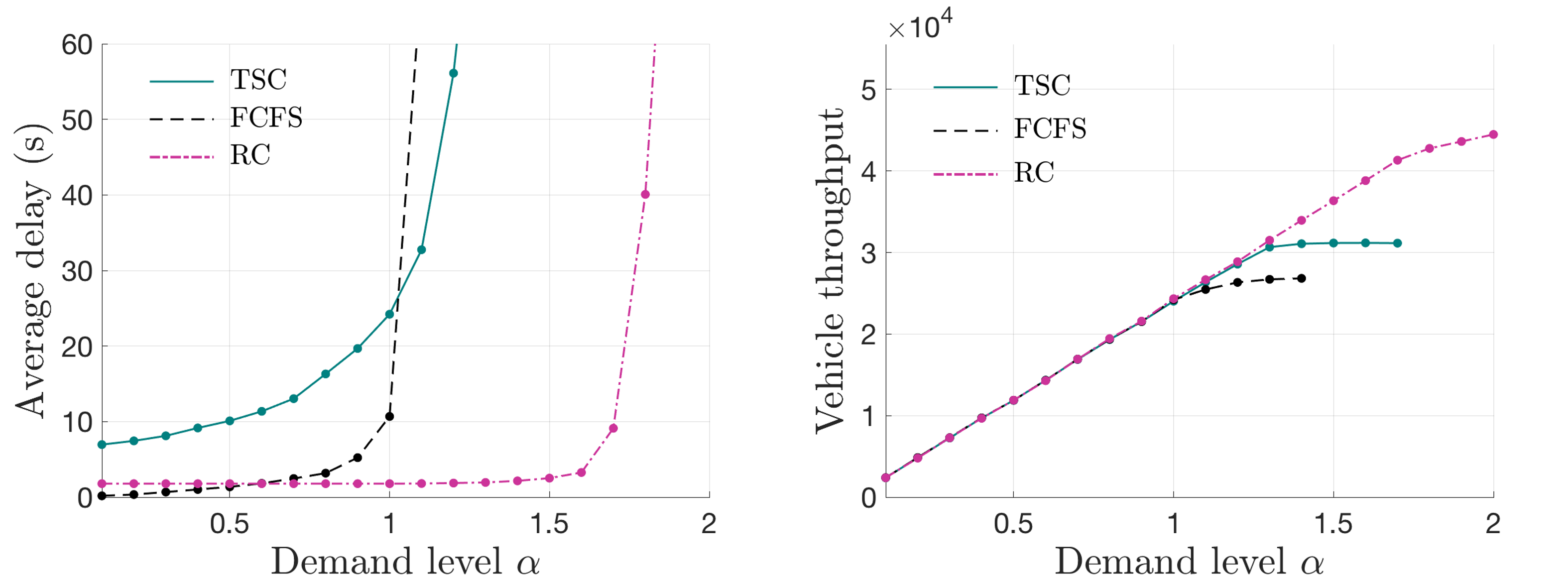}}
	\\
	\subfloat[][Imbalanced demand scenario]{\includegraphics[width=0.8\textwidth]{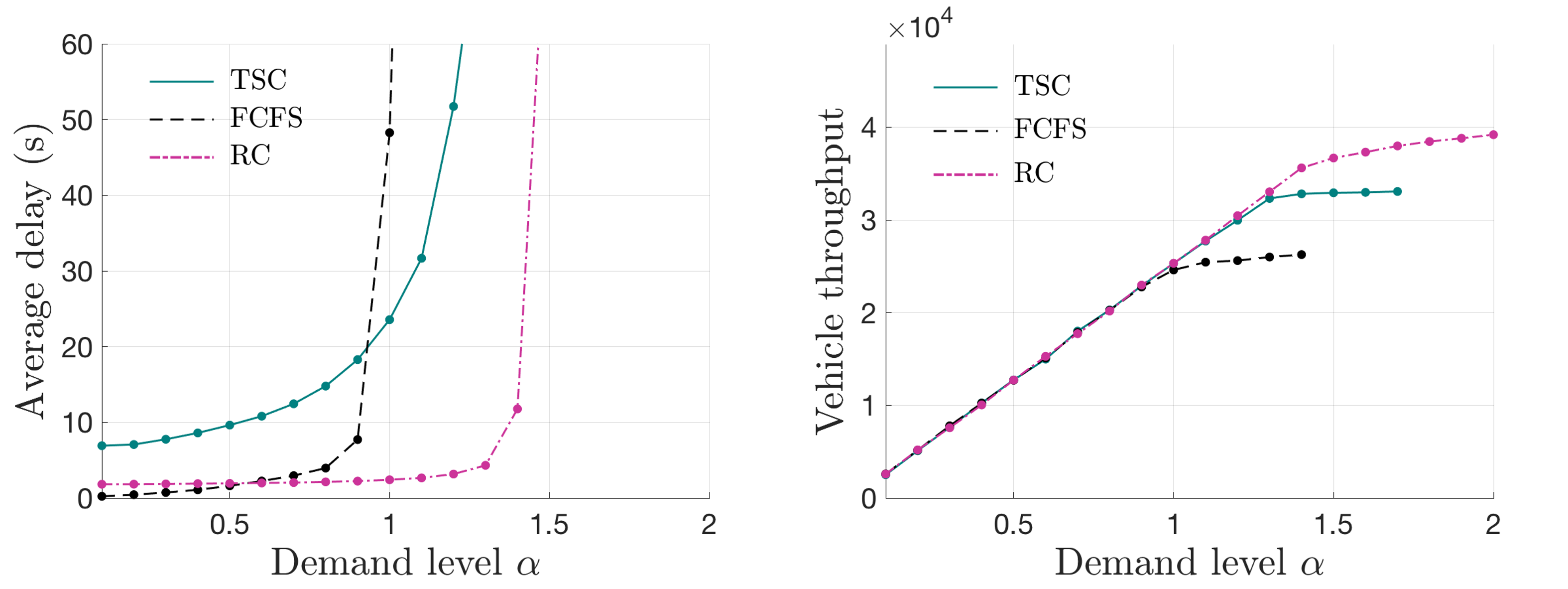}}
	\\	
	\subfloat[][Highly imbalanced demand scenario]{\includegraphics[width=0.8\textwidth]{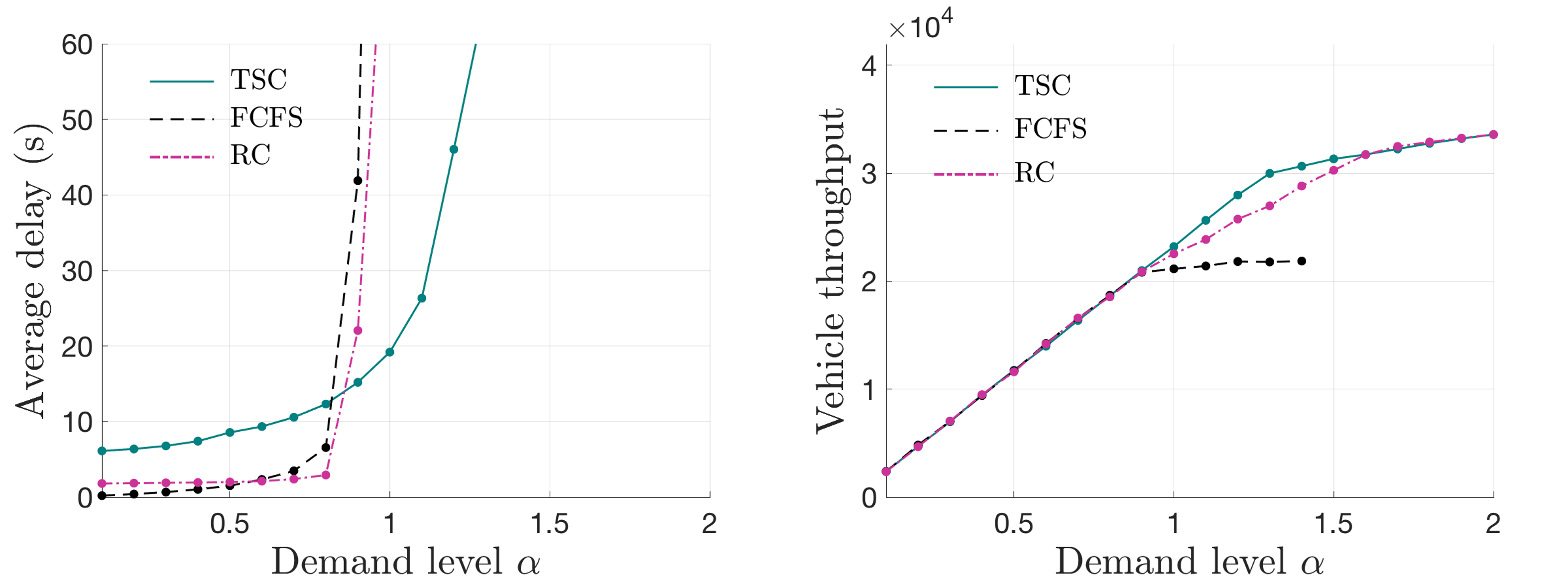}}
	\\
	\caption[]{Comparison of results with stationary vehicle arrivals}
	\label{stat_1_fig}
\end{figure}

The major observations and implications that can be derived from Fig.\ref{stat_1_fig} are summarized as follows. 
\begin{enumerate}
\item For low-demand scenarios (e.g., $\alpha \le{0.6}$), both FCFS and RC are all capable of achieving a practically zero delay. This is a major advantage for automated intersection control schemes as claimed by several publications in literature. Note that when the demand is extremely low (e.g., $\alpha \le{0.2}$), the average delay of FCFS is slightly lower than that of RC; this is attributed to the fact that a systematic delay of $1 s$ is added to RC. 
\item As $\alpha$ increases, RC is the last to become over-saturated (i.e., the average delay is unacceptably high) and thus achieves the maximum throughput among the three control schemes in the balanced and imbalanced demand scenarios. That reflects the superior performance of RC on the intersection capacity. Meanwhile, it is observed that the FCFS exhibits the lowest admissible demand set among all control schemes implemented in this test. 
\item In the comparison between TSC and RC, it is first observed that the average vehicle delay of RC is considerably smaller than that of TSC. Moreover, the range of $\alpha$ corresponding to the under-saturation intersection under RC is distinctly larger than that under TSC in the balanced and imbalanced demand scenarios. 
\item As the demand pattern becomes more imbalanced, the advantage of RC over TSC becomes less significant, i.e., the difference between the ranges of $\alpha$ corresponding to the under-saturation intersection under RC and TSC reduces. In the highly imbalanced demand scenario, RC's range of $\alpha$ even becomes smaller than TSC's. The possible reason is that, although RC is capable of breaking the limitation that ROW can only be allocated to non-conflicting movements at a time, its ROW allocation tends to be balanced among different lanes, as demonstrated in Fig.\ref{fig_ROWRC}. Therefore, as the demand pattern becomes imbalanced, the utilization rate of the ROW allocated by RC will be compromised. Note that in the highly imbalanced demand scenario, when the average delay of RC rapidly increases with the demand, the vehicle throughput can still keep increasing. That is because the ROW allocated to the low-demand movement is not fully utilized when the demand level is low, leaving potential for handling more vehicles in the view of throughput. In fact, the imbalance issue of RC could be resolved by appropriately designing the lane number for each approach at the planning stage of an intersection, suggesting that integrating the design of an intersection with the RC scheme can further enhance its performance. 
\end{enumerate}

\subsubsection{Non-stationary vehicle arrivals}\label{subsubsec_nonstationary}
\begin{figure}[!ht]
	\centering
	\subfloat[][Balanced demand scenario]{\includegraphics[width=0.8\textwidth]{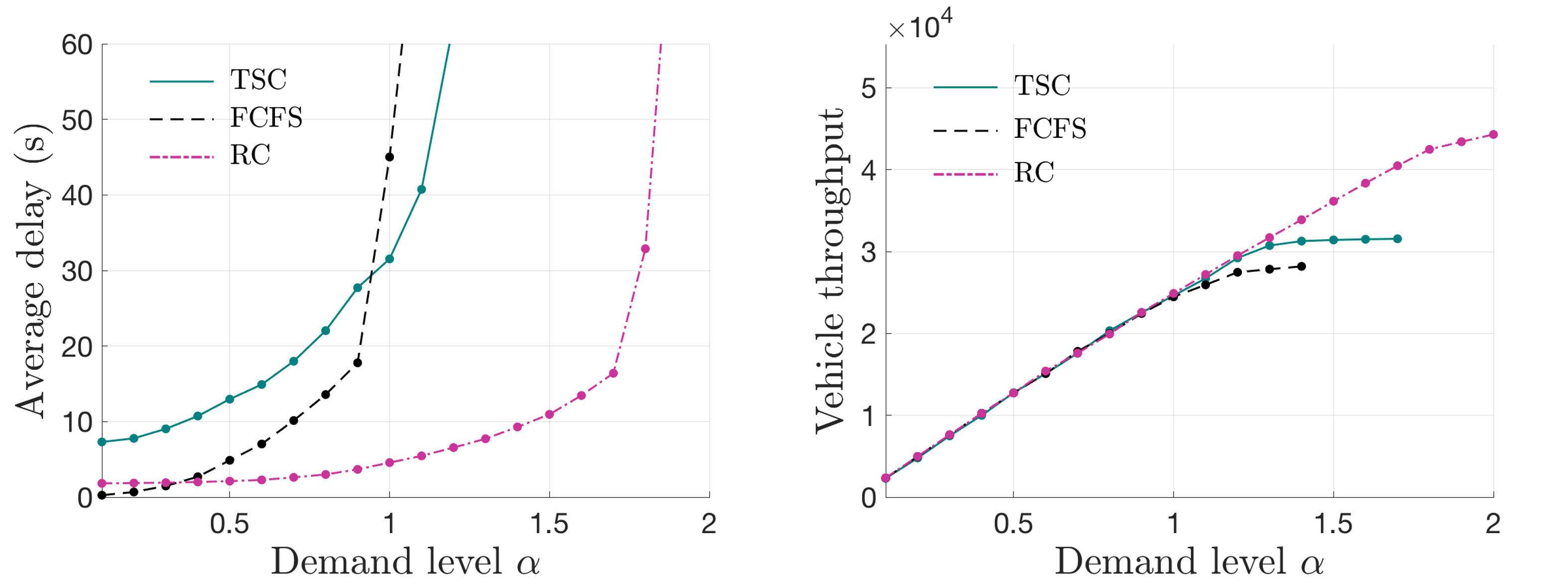}}
	\\
	\subfloat[][Imbalanced demand scenario]{\includegraphics[width=0.8\textwidth]{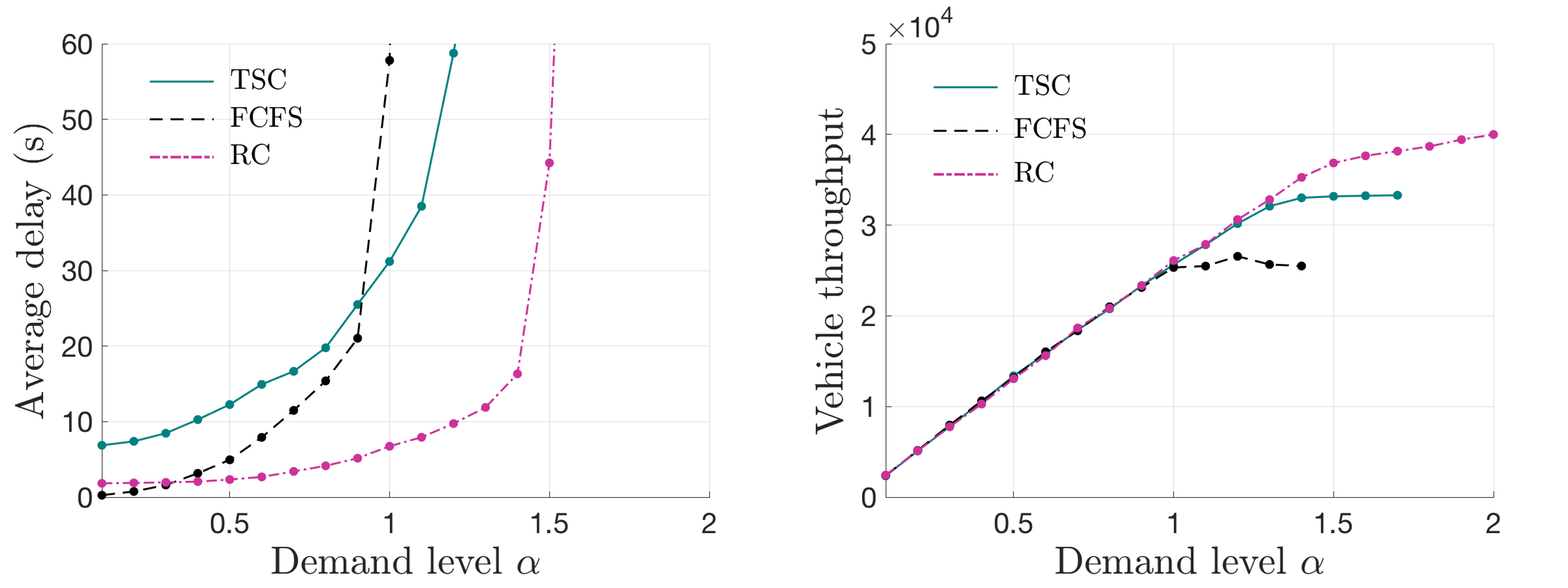}}
	\\	
	\subfloat[][Highly imbalanced demand scenario]{\includegraphics[width=0.8\textwidth]{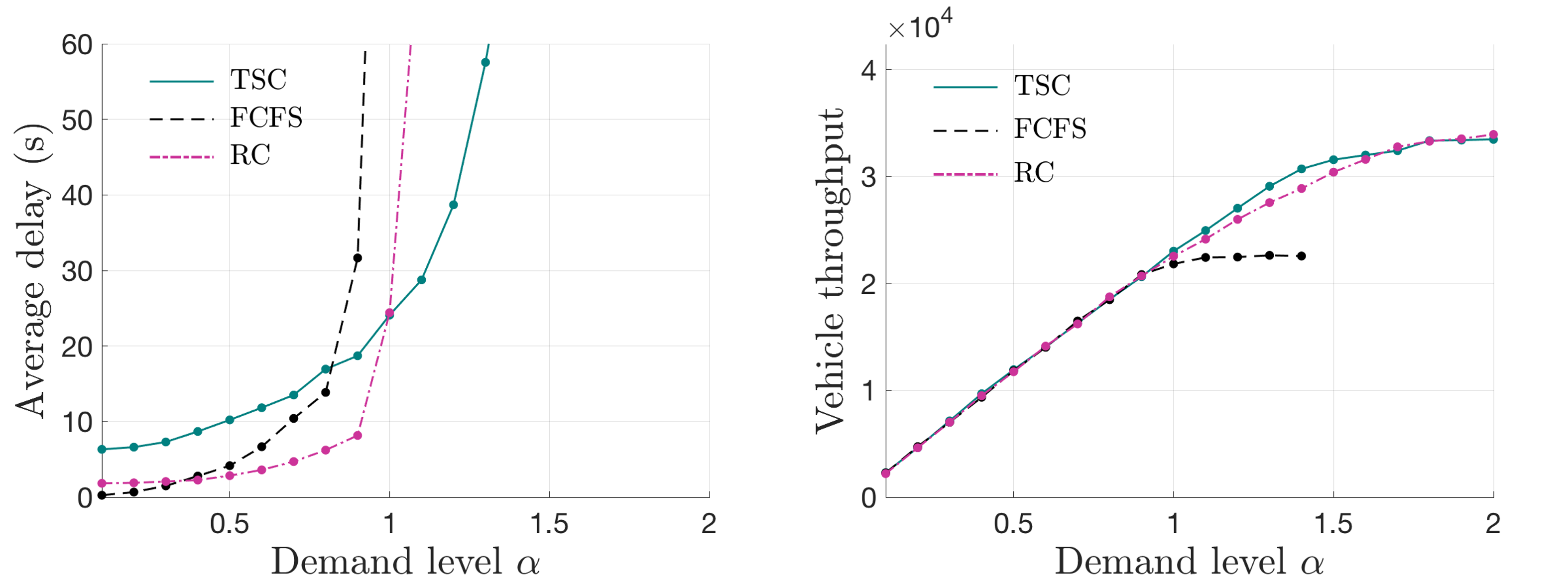}}
	\\
	\caption[]{Comparison of results with non-stationary vehicle arrivals} 
	\label{nonstat_1_fig}
\end{figure}
\noindent
Real traffic, especially on urban roads, typically exhibits non-stationary dynamics. In the following tests, the intensities of demands in all lanes are set to oscillate periodically in order to compare performances in non-stationary vehicle arrivals. Specifically, for every $200s$, there are $50s$ of “intensive arrival” and $150s$ of “mild arrival,” where the former exhibits quadruple intensity compared to the latter. The three 
scenarios, i.e., $\mathbf{d}_b$, $\mathbf{d}_i$ and $\mathbf{d}_h$, are still adopted. Figs.\ref{nonstat_1_fig}(a)–(c) illustrate the comparisons of performances with non-stationary vehicle arrivals.

Fig.\ref{nonstat_1_fig} shows that the comparison results are similar to those of the stationary arrival. It is observed that the average delays of TSC, FCFS and RC increase faster in the non-stationary arrival than in the stationary arrival. For example, when $\alpha=1.5$ in the balanced demand scenario, RC causes an average delay of approximately $3s$ in the stationary arrival and $10s$ in the non-stationary scenario; a similar phenomenon can be observed in imbalanced demand scenarios. This is because in the non-stationary arrival, a proportion of vehicle arrivals is more clustered; this results in temporal queues, which increase the average delay.

\subsection{Tests at an asymmetric intersection} \label{subsec_testonasym}

\begin{figure}[!ht]
	\centering
	\subfloat[][Balanced demand scenario]{\includegraphics[width=0.8\textwidth]{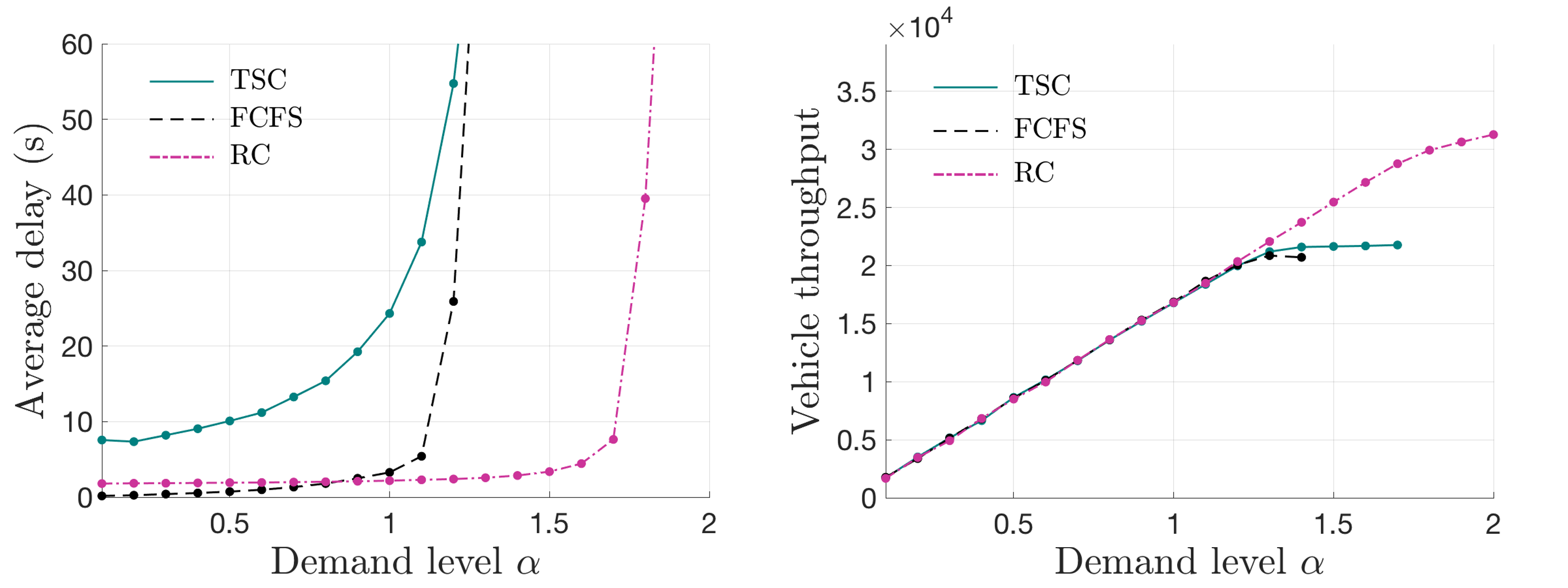}}
	\\
	\subfloat[][Imbalanced demand scenario]{\includegraphics[width=0.8\textwidth]{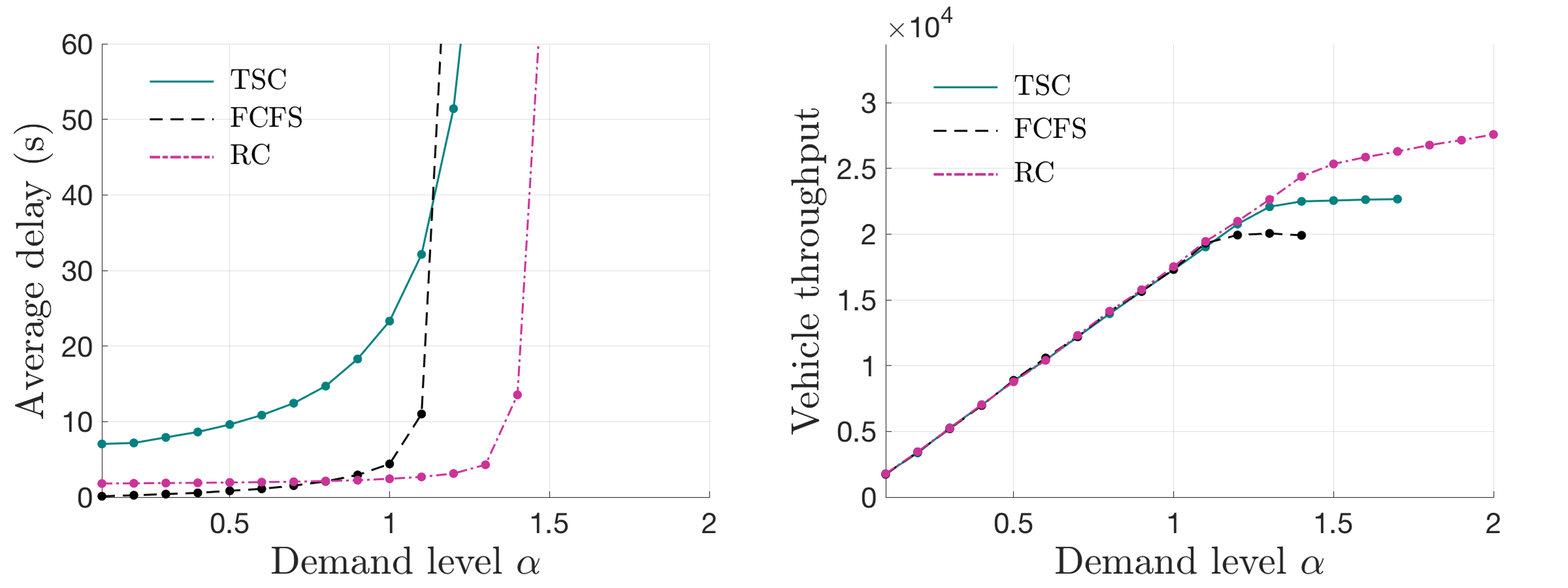}}
	\\	
	\subfloat[][Highly imbalanced demand scenario]{\includegraphics[width=0.8\textwidth]{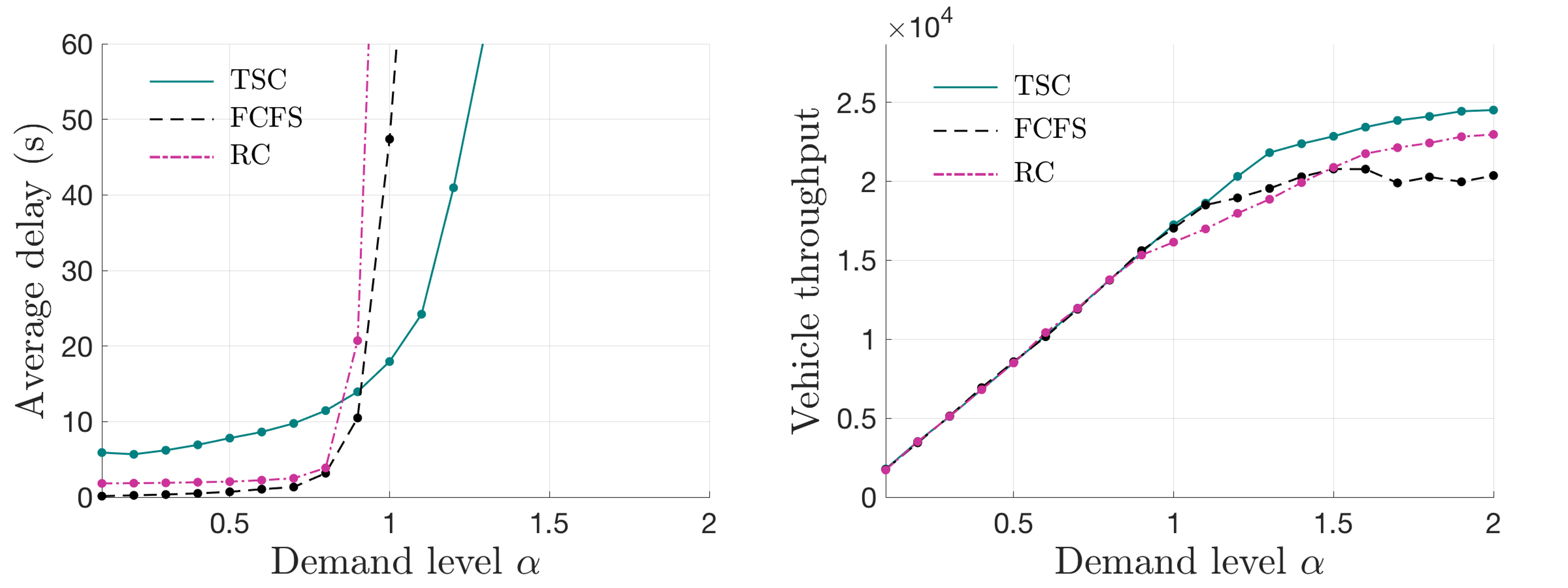}}
	\\
	\caption[]{Comparison of results at an asymmetric intersection} 
	\label{asy_1_fig}
\end{figure}

\noindent
To demonstrate the effectiveness of RC at an asymmetric intersection, simulation experiments are conducted on a four-leg intersection: legs 1 and 3 have three through lanes and two left-turn lanes, whereas legs 2 and 4 have one through lane and one left-turn lane. The vehicle arrival is set to be stationary, and the scenarios are the same as that in symmetric tests described above; results are shown in Fig.\ref{asy_1_fig}.

The performances of RC in this asymmetric intersection are considerably similar to those in the symmetric intersection; this proves the applicability of the proposed control schemes in more general settings. Interestingly, it is observes that the capacity of FCFS in this asymmetric intersection slightly improved compared to that in the symmetric intersection. One possible reason is that the simplified conflicting relationship at the asymmetric intersection provides higher possibilities for vehicles in the FCFS to find feasible time windows to pass.

\section{Conclusions and discussions} \label{sec_conclusion}
\noindent
This paper proposes an innovative intersection control scheme, the rhythmic control (RC), in a fully CAV environment. Utilizing an appropriately designed layout of conflicting points at an intersection, the RC assigns predetermined time spots in a rhythmic way for vehicles entering intersection from each lane. It can fully coordinate the complicated conflicting relationships at the intersection under given conditions. Consequently, the real-time computational load is considerably alleviated with such a predetermined rhythm. The properties of RC, the average vehicle delay and admissible demand sets under stochastic arrivals, are investigated. Finally, extensive numerical experiments are performed in various demand and vehicle arrival scenarios. The results show that RC, even when it is compared to the advanced TSC in CAV environment, can significantly reduce vehicle delays and increase intersection capacity. Among the considered three control schemes, TSC in CAV environment, FCFS, and RC, the proposed RC owns the largest admissible demand set and introduces the lowest average vehicle delay in most scenarios except for the highly imbalanced demand scenario.\par

As mentioned above, the proposed RC focuses on the traffic management at an isolated intersection, and it is more suitable to be implemented in places where major roads intersect. For dense urban areas, we further propose grid-network rhythmic control framework of CAVs, where the coordination between intersections is required to promote the mobility of the whole network. As the study focuses and control details are quite distinct, we arrange the study on grid networks as a companion of this study (Part II); the detail can be found in Lin et al. (2020).\par

In practical implementations, there are always some unexpected control errors and vehicle dysfunctions that may threaten the safe operations of RC. Thus, we need to design backup control protocols for handling emergencies. Below we discuss some feasible remedies  to handle different situations:
\begin{enumerate}[(i)]
    \item  In most cases, the vehicle deviates from its planned trajectory within a limited range of error due to small disturbances on vehicle dynamics, and for this case, a minimum allowable distance is set between any two vehicles for ensuring safety, i.e., $\delta$, and integrated into the collision-free conditions. Therefore, within the range of $\delta$, vehicles won’t collide with each other. Also, with some longitudinal control methods, such as proportional integral derivative (PID) control and model predictive control, vehicles will catch up with the planned trajectory soon before further cumulative errors.
    \item  When the deviation of trajectory is beyond the range of $\delta$, the vehicle should be adjusted to other temporal trajectories, that is, advanced or delayed. As the entrance time points are fixed for each lane, it is simple to find a suitable trajectory for the vehicle according to its actual position and speed. Fig.\ref{traj_adjust_fig} shows an example where the vehicle deviates from the objective trajectory (red dotted line), and then it can be assigned to the adjacent feasible trajectory (red solid line); thus, no collision will occur as long as the vehicle arrives at the expected time point. Note that if there are other vehicles on the same lane, their trajectories may also be required to adjust correspondingly for avoiding longitudinal collision. However, if the problematic vehicle cannot catch up with the designed speed, then an alternative solution is to adjust the pace of RC to keep all other vehicles' paces with the problematic vehicle, so that all vehicles will slow down to the same pace to ensure collision-free condition.\par

   \begin{figure}
    \centering
    \includegraphics[width=0.6\textwidth]{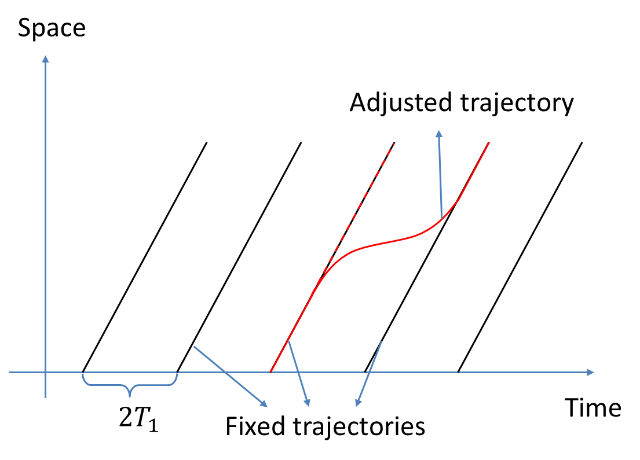}
    \caption{An example for trajectory adjustment}
    \label{traj_adjust_fig}
    \end{figure}
    \item For more extreme cases where a vehicle breaks down or an accident occurs, causing blocks in the intersection zone, adjusting vehicles’ trajectories within a small range is incapable of handling the situation. A simple solution is to turn off the availability of the affected lanes temporarily until the block has been removed. Fig.\ref{lane_turnoff_fig} shows an example where the block appears on an intersecting point of two lanes; thus, the two relevant lanes will be turned off temporarily with no impact on the traffic of other lanes. Note that here we just propose a feasible solution to handle the situation of blocks appearing. It is possible to develop more sophisticated strategies that can further reduce the impacts on the system, or can help the system recover from the malfunction soon. This will be valuable yet non-trivial task. We thus leave it for future investigation.\par 

    \begin{figure}
    \centering
    \includegraphics[width=0.5\textwidth]{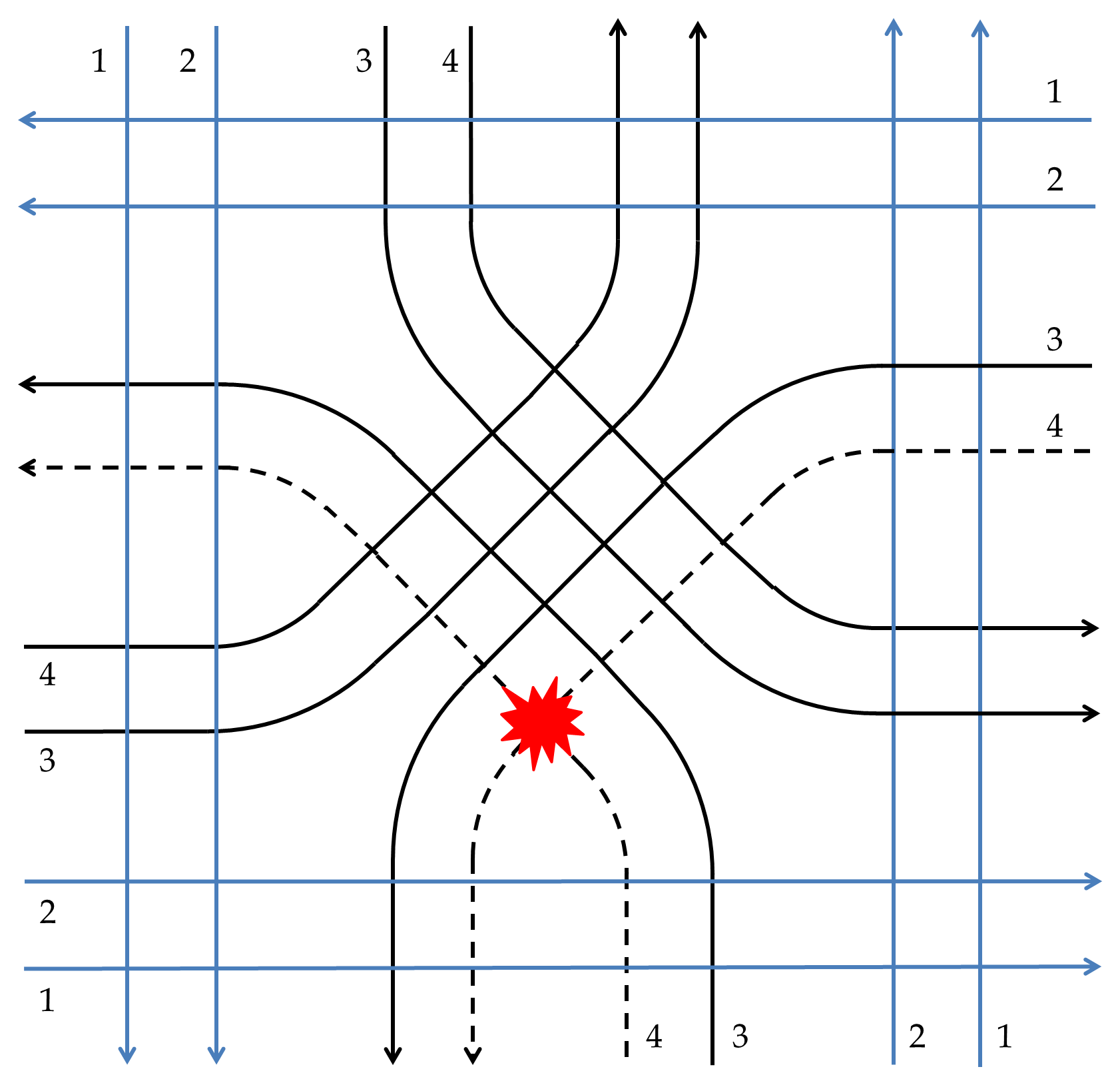}
    \caption{An example for lane turn-off}
    \label{lane_turnoff_fig}
    \end{figure}

\end{enumerate}
There are also extensive further investigations following the line of the proposed RC framework. Demonstrating the effectiveness of the proposed scheme at a real-world intersection is the first step. Further exploration is needed to accommodate RC to a more general intersection, where the layout is in the current form and has shared left-turning and through lanes. Furthermore, extremely imbalanced demand situations shall be handled when the design of an intersection layout possesses more freedom. Finally, besides proposing an innovative intersection traffic control method, this study also brings about the idea of "changing the form of land use in CAV era to support a better control method", and that could be very useful in guiding the design of future cities.

\section*{Acknowledgements}
\noindent This research is supported partially by grants from the National Natural Science Foundation of China (71871126, 51622807). Yafeng Yin thanks the support from the National Science Foundation (CNS-1837245 and CMMI-1904575) and Center for Connected and Automated Transportation (CCAT) at University of Michigan. This paper benefits significantly from the discussions with Prof. Zuo-jun Max Shen at University of California, Berkeley.

\section*{References}
\noindent Dresner, K. and Stone, P. (2004). Multiagent traffic management: A reservation-based intersection
control mechanism. In Proceedings of the Third International Joint Conference on Autonomous Agents
and Multiagent Systems, 2:530–537.\par 
\noindent Fajardo, D., Au, T. C., Waller, S., Stone, P., and Yang, D. (2011). Automated intersection control:
performance of future innovation versus current traffic signal control. Transport. Res. Rec.: J.
Transport. Res. Board, 2259:223–232.\par 
\noindent Feng, Y., Yu, C., and Liu, H. X. (2018). Spatiotemporal intersection control in a connected and
automated vehicle environment. Transportation Research Part C: Emerging Technologies, 89:364–
383.\par 
\noindent Immers, L. H. and Logghe, S. (2002). Traffic flow theory. faculty of engineering. Department of
Civil Engineering, Section Traffic and Infrastructure, 40:21.\par 
\noindent Lee, J. and Park, B. (2012). Development and evaluation of a cooperative vehicle intersection
control algorithm under the connected vehicles environment. IEEE Transactions on Intelligent
Transportation Systems, 13(1):81–90.\par 
\noindent Lee, T. T. (1989). M/g/1/n queue with vacation time and limited service discipline. Performance
Evaluation, 9(3):181–190.\par 
\noindent Levin, M. W., Boyles, S. D., and Patel, R. (2016). Paradoxes of reservation-based intersection
controls in traffic networks. Transportation Research Part A: Policy and Practice, 90:14–25.\par 
\noindent Levin, M. W. and Rey, D. (2017). Conflict-point formulation of intersection control for autonomous
vehicles. Transportation Research Part C: Emerging Technologies, 85:528–547.\par 
\noindent Li, Z., Chitturi, M., Zheng, D., Bill, A., and Noyce, D. (2013). Modeling reservation-based autonomous
intersection control in vissim. Transport. Res. Rec.: J. Transport. Res. Board, 2381:81–90.\par 
\noindent Li, Z., Elefteriadou, L., and Ranka, S. (2014). Signal control optimization for automated vehicles at
isolated signalized intersections. Transportation Research Part C: Emerging Technologies, 49:1–18.\par 
\noindent Lin, X., Li, M., Shen, Max, Z.-j., Yin, Y., and He, F. (2020). Rhythmic control of automated traffic
- part ii: Grid network rhythm and online routing. Transportation Science, under review\par 
\noindent Little, J. D. (1961). A proof for the queuing formula: L= l w. Operations research, 9(3):383–387.\par 
\noindent Muller, E. R., Carlson, R. C., and Junior, W. K. (2016). Intersection control for automated vehicles
with milp. IFAC-PapersOnLine, 49(3):37–42.\par 
\noindent NHTSA (2013). Preliminary statement of policy concerning automated vehicles. Technical report.\par 
\noindent Rajamani, R. and Shladover, S. E. (2001). An experimental comparative study of autonomous and
co-operative vehicle-follower control systems. Transportation Research Part C: Emerging Technologies,
9(1):15–31.\par 
\noindent Webster, F. V. (1958). Traffic signal settings. Technical report.\par 
\noindent Wu, J., Abbas-Turki, A., and El Moudni, A. (2012). Cooperative driving: an ant colony system
for autonomous intersection management. Applied Intelligence, 37(2):207–222.\par 
\noindent Xu, B., Li, S. E., Bian, Y., Li, S., Ban, X. J., Wang, J., and Li, K. (2018). Distributed conflictfree
cooperation for multiple connected vehicles at unsignalized intersections. Transportation
Research Part C: Emerging Technologies, 93:322–334.\par 
\noindent Yu, C., Feng, Y., Liu, H. X., Ma, W., and Yang, X. (2018). Integrated optimization of traffic
signals and vehicle trajectories at isolated urban intersections. Transportation Research Part B:
Methodological, 112:89–112.\par 
\noindent Yu, C., Sun, W., Liu, H. X., and Yang, X. (2019). Managing connected and automated vehicles at
isolated intersections: From reservation-to optimization-based methods. Transportation Research
Part B: Methodological, 112:416–435.\par 

\newpage
\setcounter{figure}{0}
\renewcommand\thefigure{A-\arabic{figure}}
\setcounter{equation}{0}
\renewcommand{\theequation}{A-\arabic{equation}}
\setcounter{table}{0}
\section*{Appendix A} \label{appendix_notation}
\begin{table}[!ht]
  \centering
    \caption*{Notations}
  \label{Nomenclature_table}
  \normalsize
  \begin{tabular}{ l  l }
  \hline
        $n_s$ & Number of through lanes\\
        $n_l$ & Number of left-turn lanes\\
        $\mathcal{S}_i$ & Set of segments of Category $i$, $i=1,2,3,4$\\
        $\mathcal{S}_5^l$ & Set of segments of Category $5$ on lane number $l$, $l\in\{n_{s+1},…,n_s+n_l\}$ \\       
        $T_i$ & Vehicle travel time on segment in  $\mathcal{S}_i$, $i=1,2,3,4$\\
        $T_5^l$ & Vehicle travel time on segment in $\mathcal{S}_5^l$, $l\in\{n_{s+1},…,n_s+n_l\}$ \\
        $T$ & The minimum time gap of two vehicles consecutively passing through the conflicting point\\
        $L$ & Vehicle length\\
        $w$ & Vehicle width\\
        $\delta$ & The minimum allowable distance between any two vehicles for ensuring safety\\
        $v_m$ & Maximum speed of vehicles\\
        $a_m$ & Maximum absolute value of acceleration and deceleration rates of vehicles\\
        $\theta$ & Rate of vehicle arrivals\\
        $\hat{D}$ & Average delay for vehicles under RC\\
        $\theta_i$ & Demand rate per lane for approach $i$, $i=1,2,…,8$\\
        $\alpha$ & Scaling factor that represents the traffic demand level\\
        $s$ & Length of the adjustment zone\\
        $v_q$ & Specific vehicle speed used in the trajectory construction procedure\\
\hline
  \end{tabular} \\
\end{table}

\newpage
\setcounter{figure}{0}
\renewcommand\thefigure{B-\arabic{figure}}
\setcounter{equation}{0}
\renewcommand{\theequation}{B-\arabic{equation}}
\setcounter{table}{0}

\section*{Appendix B} \label{appendix_proof}
\noindent
This appendix is for the proof of Propositions.\par

\subsection*{Proof of Proposition \ref{Prop_collisionfree}}
It is necessary to prove that for any conflicting point at an intersection, the headway between any two consecutive vehicles reaching this point is no less than $T_1$ because $T_1=T$, as shown by Condition (1). To achieve this, the conflicting points are classified into four types, as shown in Fig. \ref{categoryquan_fig}; then, each type is investigated. \par

\begin{figure}[!ht]
	\centering
	\includegraphics[width=0.7\textwidth]{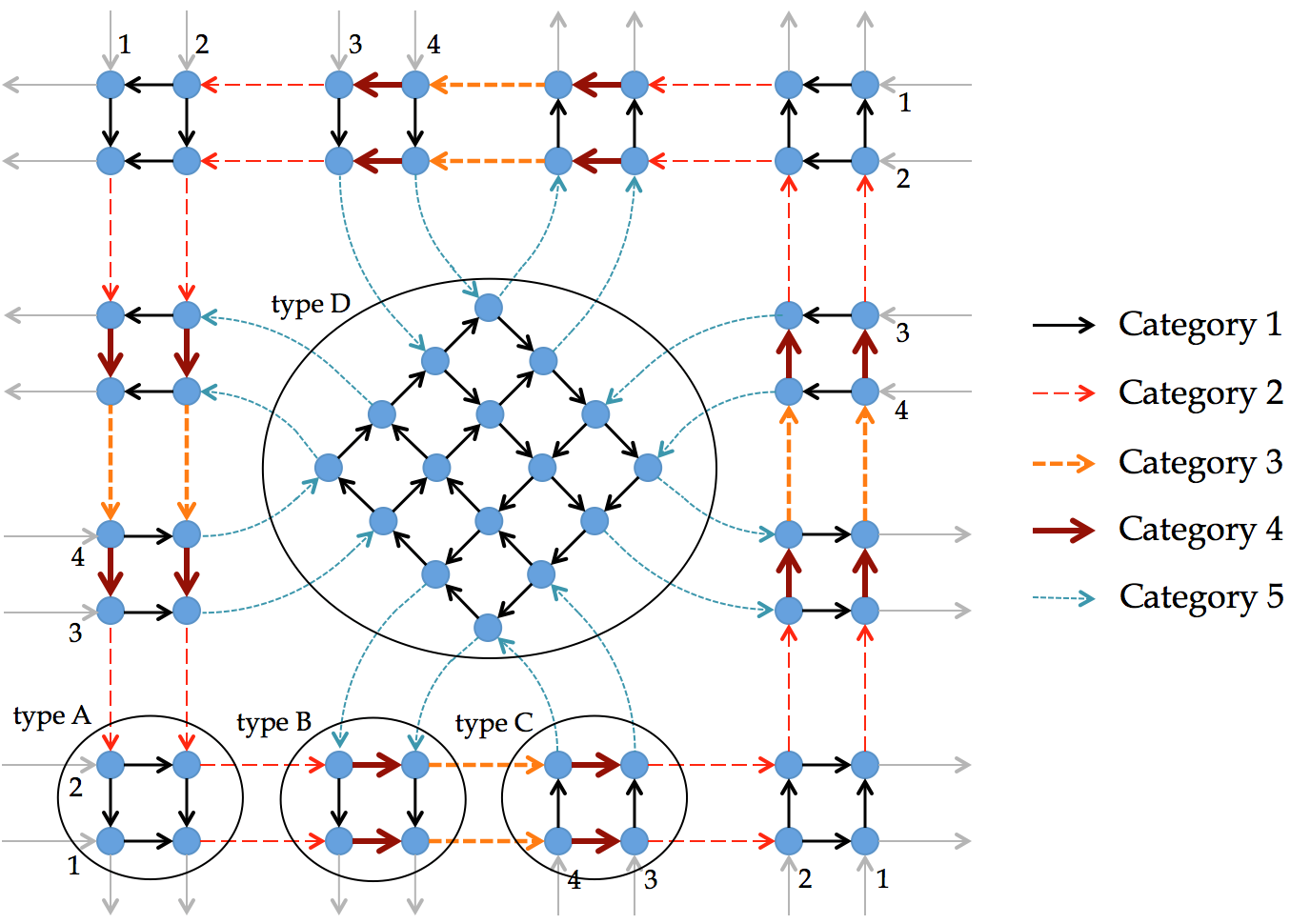}
	\caption[]{Categories of conflict points}
	\label{categoryquan_fig}
\end{figure}
\par

For simplicity, we re-express the intersection-entry time points of vehicles under the RC scheme, that is, vehicles on through lane $l\in\{{1,2,…,n_s}\}$ enter the intersection at time $t=(2k+l)T_1,k\in \mathbb{Z}$, and vehicles on left-turn lane $\hat{l}\in\{{n_s + 1,...,n_s + n_l}\}$ enter the intersection at time $t = (2k + n_s - 1)T_1+(2n_l + n_s- \hat{l} - 1) T_4+T_2+T_3,  k\in\mathbb{Z} $, which is equivalent to the previous expression of Section \ref{subsec_RCdescrib}.

\textbf{Type A} conflicting points are those with two through lanes, $l_1\in\{1,2,…,n_s\}$ and $l_2\in\{1,2,…,n_s\}$, from the conflicting legs. Only the point located at the back end of lane $l_1$ and front end of lane $l_2$ from the conflicting legs is investigated, as shown in Fig. \ref{illustrationl1l2A_fig}. For the other nodes of Type A, the following analysis can be similarly conducted. \par

\begin{figure}[!ht]
	\centering
	\includegraphics[width=0.4\textwidth]{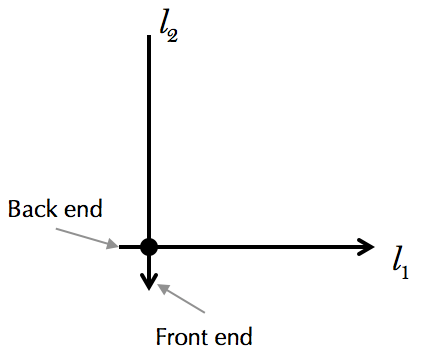}
	\caption[]{Illustration of $l_1$ and $l_2$}
	\label{illustrationl1l2A_fig}
\end{figure}
\par

The arrival times of vehicles on lanes $l_1$ and $l_2$ at the conflicting point are calculated as follows:
\begin{align*}
T_{l_{1}}&=(2k_{1}+l_{1})T_{1}+(l_{2}-1)T_{1}=(2k_{1}+l_{1}+l_{2}-1)T_{1} & k_1\in \mathbb{N}\\
T_{l_2}&=(2k_2+l_2)T_1+2T_2+T_3+2(n_l-1)T_4+(2n_s-1-l_1)T_1 & \\
&=(2k_2+2n_s+2k_0+l_2-l_1)T_1 & k_0,k_2\in \mathbb{N}
\end{align*}

The second equation above is attributed to Conditions (2)–(3). Then, it is further derived that $T_{l_1}-T_{l_2}=[2(k_1-k_2-n_s-k_0+l_1)-1]T_1, k_0,k_1,k_2\in \mathbb{N}.$ It is observed that $2(k_1 - k_2 - n_s - k_0 + l_1 ) - 1$ is an odd number; thus, we derive that $\abs{T_{l_1}-T_{l_2}}\geq T_1$; this means that the headway between any two consecutive vehicles reaching Type A conflicting points is no less than $T_1$.\par

\textbf{Type B} conflicting points are those with one through lane, $l_1\in \{1,2,…,n_s\}$, and one left-turn lane, $l_2\in \{n_s+1,…,n_s+n_l\}$, from the conflicting legs. The point located at the front end of lane $l_2$, as shown in Fig. \ref{illustrationl1l2B_fig}, is taken as an example for the investigation.

\begin{figure}[!ht]
	\centering
	\includegraphics[width=0.3\textwidth]{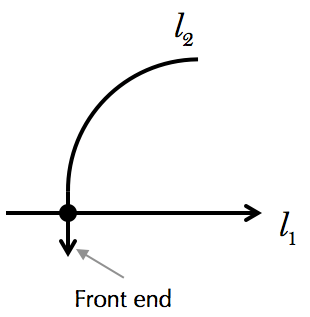}
	\caption[]{Illustration of $l_1$ and $l_2$}
	\label{illustrationl1l2B_fig}
\end{figure}
\par

The arrival times of vehicles on lanes $l_1$ and $l_2$ at the conflicting point are calculated as follows: 
\begin{align*}
T_{l_{1}}=&(2k_1+l_1+n_s-1)T_1+T_2+(l_2-n_s-1)T_4 & k_1\in \mathbb{N}\\
T_{l_2}=&(2k_2+3n_s+2n_l-l_1-3)T_1+(2n_l+n_s-l_2-1) T_4 & \\
&+T_2+T_3+2T_5^{l_2}  & k_2\in \mathbb{N}
\end{align*}

Moreover, we have: 
\begin{align*}
T_{l_{1}}-T_{l_2}&=2(k_1-k_2+l_1-n_s-n_l+1)T_1+2(l_2-n_s-n_l)T_4+T_3+2T_5^{l_2}\\
&=2(k_1-k_2+l_1-n_s-n_l+1)T_1+2(l_2-n_s-n_l)T_4+(2k_0+1)T_1
\end{align*}

The equation above holds because of Condition (4). By Condition (2), it is known that $2(l_2-n_s-n_l)T_4=2k^{'} T_1,k^{'}\in \mathbb{N}$; thus, $T_{l_1}-T_{l_2}=[2(k_1-k_2+k^{'}+k_0+l_1-n_s-n_l)+3]T_1$. Because $2(k_1-k_2+k^{'}+k_0+l_1-n_s-n_l)+3$ is odd, $\abs{T_{l_1}-T_{l_2}}\geq T_1$ is obtained.\par

\textbf{Type C} conflicting points are those with one through lane, $l_1\in \{1,2,…,n_s\}$, and one left-turn lane, $l_2\in \{n_s+1,…,n_s+n_l\}$, from the conflicting legs. The point located at the back end of lane $l_2$, as shown in Fig. \ref{illustrationl1l2C_fig}, is taken as an example for the investigation.

\begin{figure}[!ht]
	\centering
	\includegraphics[width=0.33\textwidth]{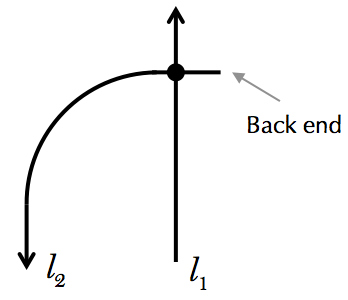}
	\caption[]{Illustration of $l_1$ and $l_2$}
	\label{illustrationl1l2C_fig}
\end{figure}
\par

The arrival times of vehicles on lanes $l_1$ and $l_2$ at the conflicting point are calculated as follows: 
\begin{align*}
T_{l_1}&=(2k_1+l_1+n_s-1)T_1+(2n_l+n_s-l_2-1)T_4+T_2+T_3 & k_1\in \mathbb{N}\\
T_{l_2}&=(2k_2+l_1+n_s-2)T_1+(2n_l+n_s-l_2-1)T_4+T_2+T_3 & k_2\in \mathbb{N}
\end{align*}
$T_{l_1}-T_{l_2}=(2k_1-2k_2+1)T_1 ,k_1,k_2\in \mathbb{N}$. Because $2k_1-2k_2+1$ is odd, $\abs{T_{l_1}-T_{l_2}}\geq T_1$ is obtained.

\textbf{Type D} conflicting points are those with two left through lanes, $l_1\in \{n_s+1,…,n_s+n_l\}$ and $l_2\in \{n_s+1,…,n_s+n_l\}$, from the conflicting legs. The point located at the back end of lane $l_1$ and front end of lane $l_2$, as shown in Fig. \ref{illustrationl1l2D_fig}, is taken as an example for the investigation.

\begin{figure}[!ht]
	\centering
	\includegraphics[width=0.35\textwidth]{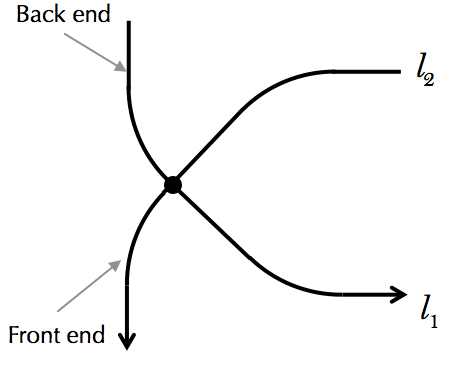}
	\caption[]{Illustration of $l_1$ and $l_2$}
	\label{illustrationl1l2D_fig}
\end{figure}
\par

\noindent The arrival times of vehicles on lanes $l_1$ and $l_2$ at the conflicting point are calculated as follows:
\begin{align*}
T_{l_1}&=(2k_1+3n_s+n_l-l_2-2)T_1+(2n_l+n_s-l_1-1)T_4+T_2+T_3+T_5^{l_1}& k_1\in \mathbb{N}\\
T_{l_2}&=(2k_2+n_s+n_l+l_1-3)T_1+(2n_l+n_s-l_2-1)T_4+T_2+T_3+T_5^{l_2} & k_2\in \mathbb{N}
\end{align*}

And, $T_{l_1}-T_{l_2}=[2(k_1-k_2+n_s)-l_1-l_2+1]T_1-(l_1+l_2)T_4+(T_5^{l_1 }-T_5^{l_2}),k_1,k_2\in \mathbb{N}$. By Condition (2) it is known that $(l_1+l_2) T_4=(l_1+l_2)(2k^{'}+1) T_1,k^{'}\in \mathbb{N}$, and by Condition (5), we have $T_5^{l_1}-T_5^{l_2}=2k^{''}T_1,k^{''}\in \mathbb{N}$. Thus, $T_{l_1}-T_{l_2}=[2(k_1-k_2+n_s)-2(l_1+l_2 )(k^{'}+1)+2k^{''}+1]T_1,k_1,k_2,k^{'},k^{''}\in \mathbb{N}$. Because $2(k_1-k_2+n_s)-2(l_1+l_2 )(k^{'}+1)+2k^{''}+1$ is odd, $\abs{T_{l_1}-T_{l_2}}\geq T_1$ is obtained.

With the above reasoning, the proof of Proposition \ref{Prop_collisionfree} is complete.\hfill $\square$

\subsection*{Proof of Proposition \ref{delay_bound_prop}}

We rewrite the state transition function as the following form:
\begin{align}
& w_{k+1} = w_k + \lambda_{k} - \mu_{k} & \forall k = 0,1,2,\dots
\end{align}

\noindent where $\lambda_k$ and $\mu_k$ are the number of vehicles entering the adjustment zone and allowed to enter the intersection, respectively; $\mu_k = \min (w_k, 1)$. Then we have:
\begin{align}
& w_{k+1}^2 - w_k^2 = (\lambda_k - \mu_k)^2 + 2(\lambda_k - \mu_k)w_k
\end{align}

And further:
\begin{align} \label{transition_proof_eq}
& \mathbb{E}(w_{k+1}^2 - w_k^2 | w_k) = \mathbb{E}(\lambda_k - \mu_k)^2 + 2 w_k \mathbb{E}(\lambda_k - \mu_k)
\end{align}

By assumption we know $0 \le \lambda_k \le 2$. When $\mu_k = 0$, $\mathbb{E}(\lambda_k - \mu_k)^2 = \mathbb{E}(\lambda_k)^2 = P_1^{\theta} + 4 P_2^{\theta} \le 2 P_1^{\theta} + 4 P_2^{\theta} = 2 \mathbb{E}(\lambda_k) = 4 \theta T_1 $. When $\mu_k = 1$, $ \mathbb{E}(\lambda_k - \mu_k)^2 = \mathbb{E}(\lambda_k - 1)^2 = P_0^{\theta} + P_2^{\theta} \le 1 $. \par

On the other hand, when $w_k = 0$, $2 w_k \mathbb{E}(\lambda_k - \mu_k) = 0 \le 2 (2 \theta T_1 - 1) w_k$; when $w_k \ge 1$, we have $2 w_k \mathbb{E}(\lambda_k - \mu_k) = 2 w_k \mathbb{E}(\lambda_k) - 2 w_k = 2 (2 \theta T_1 - 1) w_k$. \par 

In all, we always have:
\begin{align}
& \mathbb{E}(w_{k+1}^2 - w_k^2 | w_k) \le 2 (2 \theta T_1 - 1) w_k + \left\{ \begin{matrix}
4 \theta T_1 & \mathrm{if} \ \mu_k = 0 \\ 1 & \mathrm{if} \ \mu_k = 1
\end{matrix} \right.
\end{align}

Meanwhile, $\mu_k = 0$ indicates $w_k = 0$, and in the main body we have derived that $p_0^{\theta} = 1 - 2 \theta T_1$. Therefore, for a sufficiently large number $M$, by taking expectation on $w_k$ and summing over $k = 0,1,2,\dots,M$, we yield:
\begin{align}
\mathbb{E}(w_{M+1}^2) - \mathbb{E}(w_{0}^2) &\le [4 \theta T_1 (1 - 2 \theta T_1) + 2 \theta T_1] M + 2 (2 \theta T_1 - 1) \sum_{k=1}^{M} \mathbb{E}(w_{k}) \nonumber \\
&= (6 \theta T_1 - 8 \theta^2 T_1^2) M + 2 (2 \theta T_1 - 1) \sum_{k=1}^{M} \mathbb{E}(w_{k})
\end{align}

Letting $w_{0} = 0$ and taking the limit leads to:
\begin{align}
& \lim_{M \rightarrow +\infty} \frac{1}{M} \sum_{k=1}^{M} \mathbb{E}(w_{k}) \le \frac{3 \theta T_1 - 4 \theta^2 T_1^2}{1 - 2 \theta T_1} = 2 \theta T_1 + \frac{\theta T_1}{1 - 2\theta T_1}
\end{align}

Note that $\mathbb{E}(w_k)$ is the expected queue length \textit{rightly before} the $k^{th}$ entry time point; and the average queue length should average it with the expected queue length \textit{rightly after} an entry time point. So the average queue length $\overline{Q}(\theta)$ is given by:
\begin{align}
\overline{Q}(\theta) &= \lim_{M \rightarrow +\infty} \frac{1}{2M} \left\{ \sum_{k=1}^{M} \left[ \mathbb{E}(w_{k}) + \mathbb{E}(\max(w_{k}-1,0)) \right] \right\} \nonumber \\
&= \lim_{M \rightarrow +\infty} \frac{1}{2M} \left\{ \sum_{k=1}^{M} \left[ \mathbb{E}(w_{k}) + 2 \theta T_1\mathbb{E}(w_{k}-1 | w_{k} \ge 1) \right] \right\} \nonumber \\
&= \lim_{M \rightarrow +\infty} \frac{1}{M} \sum_{k=1}^{M} \left[ \mathbb{E}(w_{k}) - \theta T_1 \right]
\end{align}

The second equality is because $\mathbb{P}(w_k = 0) = 1 - 2 \theta T_1$. Since $\lim_{M \rightarrow +\infty} \frac{1}{M} \sum_{k=1}^{M} \mathbb{E}(w_{k})$ represents the average queue length, by Little's law, the average vehicle delay is then bounded by:
\begin{align}
& \overline{D}(\theta) = \frac{1}{\theta} \overline{Q}(\theta) = \lim_{M \rightarrow +\infty} \frac{1}{M\theta} \sum_{k=1}^{M} \left[ \mathbb{E}(w_{k}) - \theta T_1 \right] \le T_1 + \frac{T_1}{1 - 2\theta T_1}
\end{align}

Proof completes. \hfill $\square$

\subsection*{Proof of Proposition \ref{admissible_prop}}

We first prove that when $\theta_i < \frac{1}{2T_1}$, the time average of queue length is bounded. For brevity, in the proof we omit the approach index $i$, and the notations used are the same as in the proof of Proposition \ref{delay_bound_prop}. With similar derivation, we obtain:
\begin{align}
& \mathbb{E}(w_{k+1}^2 - w_k^2 | w_k) = \mathbb{E}(\lambda_k - \mu_k)^2 + 2 w_k \mathbb{E}(\lambda_k - \mu_k)
\end{align}

\noindent and we have already known that $2 w_k \mathbb{E}(\lambda_k - \mu_k) \le 2(2 \theta T_1 - 1)w_k$. Meanwhile, we have:
\begin{align}
\mathbb{E}(\lambda_k - \mu_k)^2 &= \mathbb{E}(\lambda_k)^2 - 2 \mu_k \mathbb{E}(\lambda_k) + \mu_k^2 = \mathbb{E}(\lambda_k)^2 - 4 \mu_k \theta T_1 + \mu_k^2
\end{align}

Since $\mathbb{E}(\lambda_k) = \sum_{l=1}^{+\infty} l^2 P_l^{\theta} < +\infty$, we acquire that $\mathbb{E}(\lambda_k - \mu_k)^2$ is finite, i.e., there exists some constant $K$ such that $\mathbb{E}(\lambda_k - \mu_k)^2 < K$. Therefore, we yield:
\begin{align}
& \mathbb{E}(w_{k+1}^2 - w_k^2 | w_k) \le K + 2(2 \theta T_1 - 1)w_k
\end{align}

For a sufficiently large number $M$, by taking expectation on $w_k$ and summing over $k = 0,1,2,\dots,M$, we yield:
\begin{align}
\mathbb{E}(w_{M+1}^2) - \mathbb{E}(w_{0}^2) &\le MK + 2(2 \theta T_1 - 1) \sum_{k=1}^{M} \mathbb{E}(w_k)
\end{align}

Letting $w_{0} = 0$ and taking the limit, we finally have:
\begin{align}
& \lim_{M \rightarrow +\infty} \frac{1}{M} \sum_{k=1}^{M} \mathbb{E}(w_{k}) \le \frac{K}{2 - 4 \theta T_1}
\end{align}

Note that $\lim_{M \rightarrow +\infty} \frac{1}{M} \sum_{k=1}^{M} \mathbb{E}(w_{k})$ can be treated as an upper bound of the time average of queue length, so the above inequality proves that $\theta_i < \frac{1}{2T_1}$ leads to a stabilized queue.\par

On the other hand, when $\theta_i \rightarrow \frac{1}{2T_1}$, the average delay is approaching infinity if the arrival process follows Poisson's process (Eq.(\ref{delay_poisson_eq})). When $\theta_i > \frac{1}{2T_1}$, the arrival rate is larger than the maximum service rate, so the queue grows to infinity. Therefore, for a general arrival process satisfying $\sum_{l=1}^{+\infty} l^2 P_l^{\theta} < +\infty$, the admissible demand set is given by $\theta_i < \frac{1}{2T_1} = \frac{v_m}{2(L+w+\sqrt{2}\delta)}$. Proof completes.
 \hfill $\square$

\newpage
\setcounter{figure}{0}
\renewcommand\thefigure{C-\arabic{figure}}
\setcounter{equation}{0}
\renewcommand{\theequation}{C-\arabic{equation}}
\setcounter{table}{0}

\section*{Appendix C}\label{appendix_inequality}
\noindent This appendix derives the inequality $T_1\geq\frac{L+w+\sqrt 2\delta}{v_m}$ . In Appendix B, it is proved that the headway between any two consecutive vehicles reaching a point is no less than $T_1$. Any two intersecting lanes are perpendicular to each other. Accordingly, a rectangular coordinate can be constructed; the vehicles on the two lanes move along the x-axis and y-axis, as shown in Fig. \ref{twovehicle_fig}. The coordinates of the centers of the two vehicles at time $t$ can be expressed as $(v_mt,0)$ and $(0,v_m(t+T'))$, where $\abs{T'}\geq T_1$. Without loss of generality, set $T'\geq T_1$. In Fig. \ref{twovehicle_fig}, it can be observed that a potential collision can occur between the front left corner of Vehicle 1 and left rear corner of Vehicle 2. The coordinate of the front left point of Vehicle 1 is $\mathbf{x_1} (t)=(v_mt+\frac{1}{2}L,\frac{1}{2}w)$, and that of the left rear point of Vehicle 2 is $\mathbf{x_2}(t)=(-\frac{1}{2}w,v_m(t+T')-\frac{1}{2}L)$. The distance between those two points is
\begin{align*}
\|\mathbf{x_1}(t)-\mathbf{x_2} (t)\|&=\sqrt{(v_mt+\frac{1}{2}L+\frac{1}{2}w)^2+(v_m (t+T')-\frac{1}{2}L-\frac{1}{2}w)^2}\\
&=\sqrt{2v_m^2t^2+2v_m^2tT'+v_m^2T'^2+\frac{1}{2}(L+w)^2-v_mT'(L+w)}
\end{align*}

Let $f(t)=2v_m^2t^2+2v_m^2tT'+v_m^2T'^2+\frac{1}{2}(L+w)^2-v_mT'(L+w)$; by the property of quadratic functions, $f(t)\geq f(-\frac{T'}{2})=\frac{T'^2}{2}v_m^2-v_mT'(L+w)+\frac{1}{2}(L+w)^2=\frac{1}{2}(v_mT'-L-w)^2$. The last expression is required to be no smaller than $\delta^2$ because of the safety constraint, $\|\mathbf{x_1}(t)-\mathbf{x_2} (t)\|\geq \delta$. Accordingly, $v_mT_1\geq L+w+\sqrt{2}\delta$, and the required inequality is derived.

\begin{figure}[!ht]
	\centering
	\includegraphics[width=0.5\textwidth]{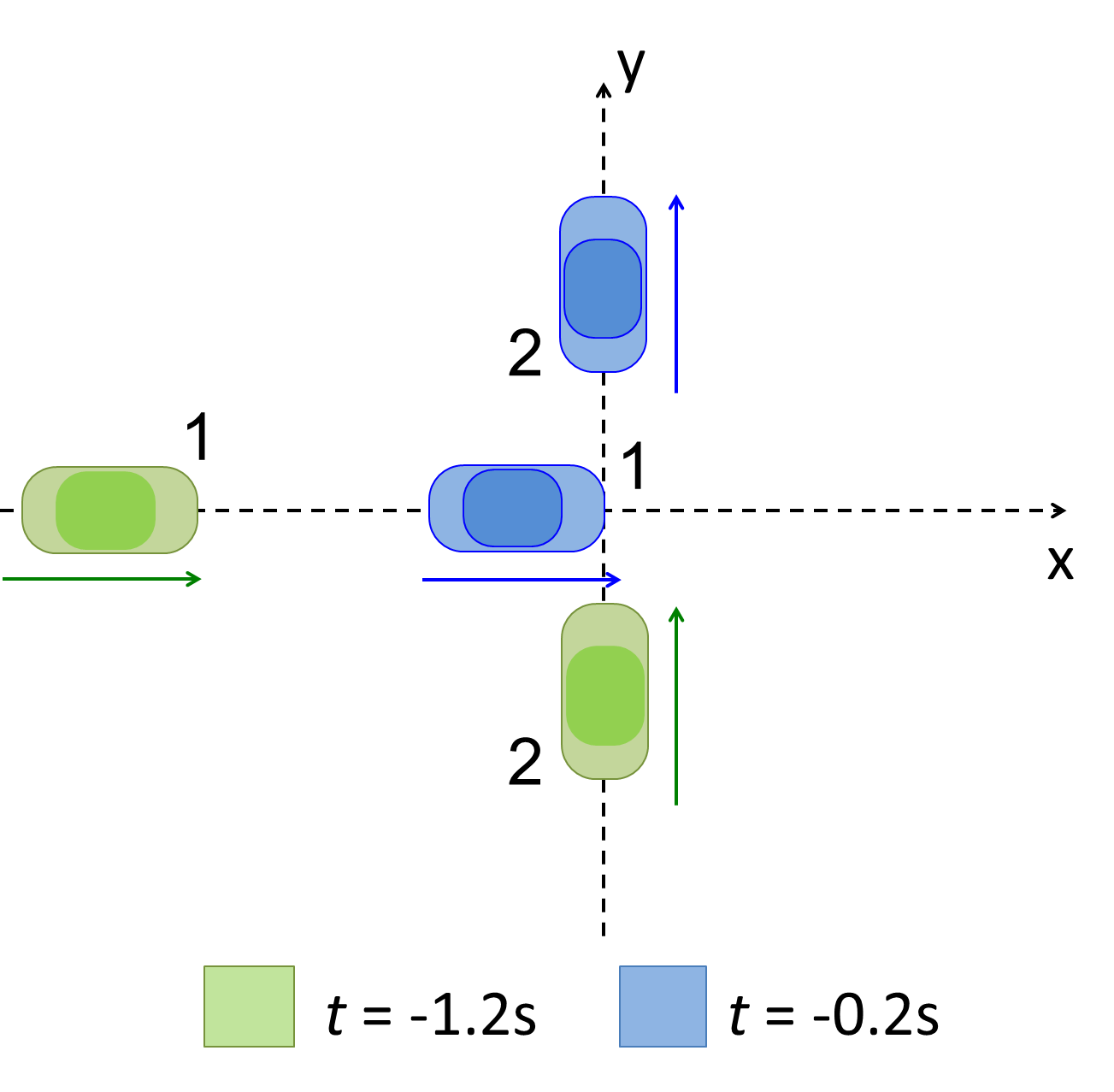}
	\caption[]{Illustration of two vehicles passing through a conflicting point}
	\label{twovehicle_fig}
\end{figure}
\par

\newpage
\setcounter{figure}{0}
\renewcommand\thefigure{D-\arabic{figure}}
\setcounter{equation}{0}
\renewcommand{\theequation}{D-\arabic{equation}}
\setcounter{table}{0}
\renewcommand{\thetable}{D-\arabic{table}}
\setcounter{prop}{0}
\renewcommand{\theprop}{D-\arabic{prop}}

\section*{Appendix D} \label{appendix_speedcurve}
\noindent In this appendix, a set of parsimonious rules for trajectory adjustment is provided; it is free from online computation and optimization.\par

Assume that the length of the adjustment zone is $s$; $x=0$ represents the location of the zone’s entry point; and $x=s$ represents its exit point. Each vehicle enters the zone with the speed $v_m$, and it is necessary to guarantee that it arrives at $x=s$ with the speed $v_m$ at the time points proposed by the RC scheme, as stated in Section \ref{subsec_CollisionAvoidance}. In the following analyses, without loss of generality, we fix the value of $v_m$ to be $\frac{L+w+\sqrt{2}\delta}{T_1}$.\par

A specific speed, $v_q$, that is used in the trajectory construction procedure is first introduced. Its value can be calculated using $v_q=\frac{L+\delta}{2T_1}$; $v_q$ is chosen such that a group of vehicles travelling at this speed is capable of crossing the adjustment zone with a minimum allowable distance, $\delta$, and the time headway between two consecutive vehicles is exactly $2T_1$, which is the minimum time headway allowed in the RC scheme for the vehicles on the same lane. Note that $v_q=\frac{L+\delta}{2T_1}$ is not the only choice; however, our observations suggest that it is an appropriate value.\par

Once a vehicle $i$ arrives at $x=0$ at time $t_i^0$, a speed curve is immediately assigned to it such that the target arrival time at $x=s$ is $t_i^s$. There are two possible types of speed curves that can be assigned; one type is shown in Fig. \ref{speedcurve_fig}. The vehicle decelerates at a rate $a_m$, over time $t_s\geq t_i^0$, and to a speed $v_q$ (at most); then, it accelerates at a rate $a_m$ and ending at time $t_e\leq t_i^s$ with a speed $v_m$. The other type of curve represents the case where a vehicle decelerates to speed $v_d>v_q$ and then immediately accelerates. 

\begin{figure}[!ht]
	\centering
	\includegraphics[width=0.6\textwidth]{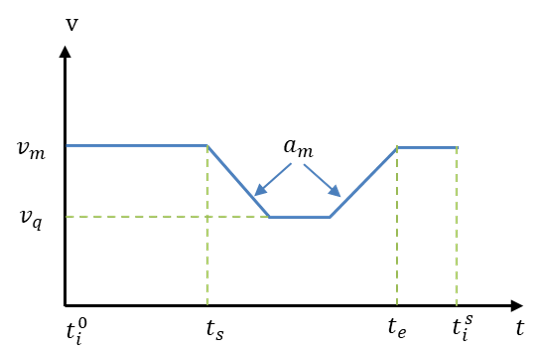}
	\caption[]{Speed curve shape}
	\label{speedcurve_fig}
\end{figure}
\par

The speed curve is denoted by $V_i(t;t_s,t_e)$ because with $t_s$ and $t_e$, the entire curve can be confirmed. A legal speed curve must satisfy the following constraints:
\begin{align}
&\int_{t_i^0}^{t_i^s}V_i(\omega;t_s,t_e)d\omega=s\label{eq_speedcurve1}\\
&\int_{t_{i-1}^0}^{T}V_{i-1}(\omega;t_s,t_e)d\omega-\int_{t_{i}^0}^{T}V_{i}(\omega;t_s,t_e)d\omega\geq L+\delta, &T\in [t_i^0,t_{i-1}^s]\label{eq_speedcurve2}
\end{align}

where Constraint (\ref{eq_speedcurve1}) requires that the distance traveled over time interval $[t_i^0,t_i^s]$ is $s$; Constraint (\ref{eq_speedcurve2}) requires that the distance between vehicles $i$ and $i-1$ (the vehicle before $i$) is always not less than $\delta$. Based on speed curve shapes, the integral of $V_i(t;t_s,t_e )$ can be calculated by $\int_{t_i^0}^{t_i^s}V_i(\omega;t_s,t_e)d\omega=\left\{
\begin{aligned}
&v_m(t_i^s-t_i^0-t_e+t_s)+v_q[t_e-t_s-\frac{2(v_m-v_q)}{a_m} ]+\frac{v_m^2-v_q^2}{a_m}, & t_e-t_s\geq \frac{2(v_m-v_q)}{a_m}\\
&v_m(t_i^s-t_i^0)-\frac{a_m(t_e-t_s)^2}{4}, & t_e-t_s<\frac{2(v_m-v_q)}{a_m}
\end{aligned}
\right.$. By introducing this equation into Constraint (\ref{eq_speedcurve1}), it can be found that once $t_i^0$ and $t_i^s$ are given, $\Delta t=t_e-t_s$ can be easily calculated using the linear equation, Eq. (\ref{eq_speedcurve3}), or the quadratic equation, Eq. (\ref{eq_speedcurve4}), as follows: 
\begin{align}
&v_m(t_i^s-t_i^0-\Delta t)+v_q[\Delta t-\frac{2(v_m-v_q)}{a_m}]=s-\frac{v_m^2-v_q^2}{a_m}\label{eq_speedcurve3}  \\
&\frac{a_m\Delta t^2}{4}=v_m(t_i^s-t_i^0)-s \label{eq_speedcurve4}
\end{align}

The following elaborates how Constraint (\ref{eq_speedcurve2}) determines $t_e$ in Proposition \ref{prop_speedcurve}. Note that $t_e^{(i)}$ and $\Delta t^{(i)}$ are used to denote $t_e$ and $\Delta t$ for Vehicle $i$, respectively. 

\begin{prop}\label{prop_speedcurve}
If $\Delta t^{(i)}\geq \Delta t^{(i-1)}$, then by setting $t_e^{(i)}=t_e^{(i-1)}$, it can be guaranteed that the distance between Vehicles $i$ and $i-1$ is always larger or equal to $\delta$. If $\Delta t^{(i)}< \Delta t^{(i-1)}$, then by setting $t_e^{(i)}=t_e^{(i-1)}+\frac{L+\sqrt{2}\delta+2w}{v_m-v_q}$, it can be guaranteed that the distance between Vehicles $i$ and $i-1$ is always larger than or equal to $\delta$.
\end{prop}

\begin{proof}
\textit{\underline{Case I $\Delta t^{(i)}\geq\Delta t^{(i-1)}$.}} Let $t_e^{(i)}=t_e^{(i-1) }$. It is evident that the distance between Vehicles $i$ and $i-1$ is non-decreasing within the time interval $[t_i^0,t_{i-1}^s]$. Because the distance between these two vehicles is not less than $\delta$ when they enter the adjusting segment, it is known that within interval $[t_i^0,t_{i-1}^s]$, the distance is always larger than or equal to $\delta$.

\textit{\underline{Case II $\Delta t^{(i)}<\Delta t^{(i-1)}$.}} Because of the RC, the distance between these two vehicles when leaving the adjusting segment should be no less than $2T_1v_m=2(L+w+\sqrt{2}\delta)$. The speed curves of these two vehicles shown in Figs. \ref{scenario_fig}(a)–(b) represent two possible scenarios. It is easy to verify that at time $t^*$, the distance between the two vehicles reaches the minimum in both scenarios, and the area of shaded zone, $S$, is the increased distance after $t^*$. By $t_e^{(i)}=t_e^{(i-1)}+\frac{L+\sqrt{2} \delta+2w}{v_m-v_q}$, the area of $S$ is no more than $(v_m-v_q)\frac{L+\sqrt{2} \delta+2w}{v_m-v_q}=L+\sqrt{2} \delta+2w$. Thus, the distance at time $t^*$ is no less than $2(L+\sqrt{2} \delta+w)-(L+\sqrt{2} \delta+2w)=L+\sqrt{2} \delta$. The proof is complete.
\end{proof}

\begin{figure}[!ht]
	\centering
	\subfloat[][Scenario 1]{\includegraphics[width=0.4\textwidth]{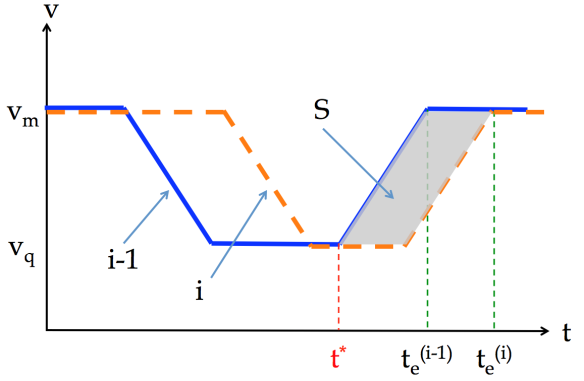}}
	\subfloat[][Scenario 2]{\includegraphics[width=0.4\textwidth]{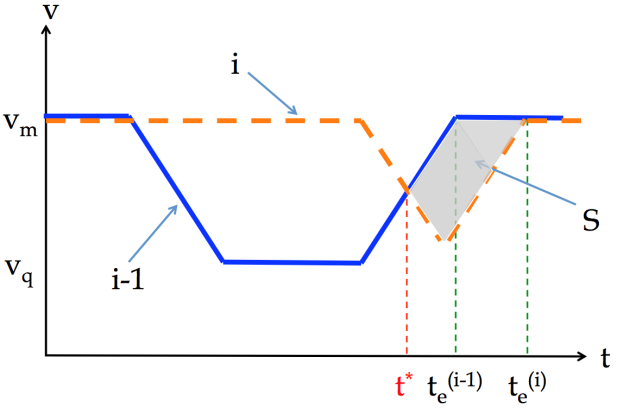}}
	\caption[]{Graphical illustration of two possible scenarios} 
	\label{scenario_fig}
\end{figure}
\par

Using Eqs. (\ref{eq_speedcurve3})–(\ref{eq_speedcurve4}) and Proposition \ref{prop_speedcurve}, the steps for generating the speed curve for Vehicle $i$ when it enters the adjustment zone are enumerated as follows.

\noindent \textbf{\underline{Speed Curve Generation}} \par
\hangafter 1
\hangindent 3.5em
\noindent \underline{\textbf{Step 0}}. Determine the earliest time point when Vehicle $i$ can arrive at $x=s$ by the RC scheme, as explained in Section \ref{subsec_CollisionAvoidance}; \par
\hangafter 1
\hangindent 3.5em
\noindent \underline{\textbf{Step 1}}. Determine the value of $\Delta t^{(i)}$ by Eqs. (\ref{eq_speedcurve3})–(\ref{eq_speedcurve4}); \par
\hangafter 1
\hangindent 3.5em
\noindent \underline{\textbf{Step 2}}. First, set $t_e^{(i)}=t_i^s$.  If Constraint (\ref{eq_speedcurve2}) holds, then fix $t_e^{(i)}$ and calculate $t_s^{(i)}$ accordingly; otherwise, determine $t_e^{(i)}$ and $t_s^{(i)}$ using Proposition \ref{prop_speedcurve}. \par

It can be observed that all the steps in the speed curve generation procedure only require elementary arithmetic operations; the foregoing is consistent with RC’s property of low computational loading.

\end{document}